\documentclass[10pt,journal,compsoc]{IEEEtran}
\ifCLASSOPTIONcompsoc
\usepackage[nocompress]{cite}
\else
\usepackage{cite}
\fi
\ifCLASSINFOpdf
\else
\fi
	\usepackage{amsmath,amsfonts}
	\usepackage{amsthm} 
	\usepackage{algorithmic}
	\usepackage{array}
	\usepackage[caption=false,font=normalsize,labelfont=sf,textfont=sf]{subfig}
	\usepackage{textcomp}
	\usepackage{stfloats}
	\usepackage{url}
	\usepackage{verbatim}
	\usepackage{graphicx}
	\newtheorem{theorem}{Theorem}
	\newtheorem{lemma}{Lemma}
	\newtheorem{remark}{Remark}

	\newtheorem{corollary}{Corollary}

	\usepackage{amssymb}
	\usepackage{amsmath}
	\usepackage{makecell}
	\usepackage{mathrsfs}
	\usepackage{algorithm}
	\usepackage{algorithmic}
	\usepackage{bm}
	\usepackage{xcolor}
	\usepackage{color}
	\usepackage{pgfplots}
	\usepackage{bbm}
	\usepackage{balance}
	\usepackage{footnote}
	\usepackage{balance}
	\usepackage{cite}
	\usepackage{cases}
	\usepackage{epstopdf}
	\usepackage{soul} %line
	\usepackage[colorlinks, linkcolor=blue,anchorcolor=blue,citecolor=blue] {hyperref}
\begin{document}
		\title{ Consensus ALADIN: A Framework for Distributed Optimization and Its Application in Federated Learning}
		
		\author{Xu Du,~\IEEEmembership{Member,~IEEE,}
			Jingzhe Wang,~\IEEEmembership{Student Member,~IEEE}
			%Xiaojun Yuan, ~\IEEEmembership{Senior Member,~IEEE}
			%	John~Doe,~\IEEEmembership{Fellow,~OSA,}
			%	and~Jane~Doe,~\IEEEmembership{Life~Fellow,~IEEE}% <-this % stops a space
			\IEEEcompsocitemizethanks{\IEEEcompsocthanksitem X.Du is with the Institute of Mathematics HNAS, Henan Academy of Science,
				Zhengzhou, China. 
			%	X.Du is also with
			%	National Key Laboratory of Science and Technology on Communications, University of
				%	Electronic Science and Technology of China, Chengdu 610000, China.\\
				E-mail:  duxu@shanghaitech.edu.cn
			}
			\IEEEcompsocitemizethanks{\IEEEcompsocthanksitem 
				Jingzhe Wang is with the Department of Informatics \& Networked Systems, School of Computing and Information, University of Pittsburgh, Pittsburgh, PA, USA. E-mail:  jiw148@pitt.edu
			}

		}% <-this % stops a space
		%	\thanks{Manuscript received April 19, 2005; revised August 26, 2015.}}
	

%\markboth{IEEE Transactions on Pattern Analysis and Machine Intelligence,~Vol., No., ~ }%
\markboth{}%
\markboth{}%
%{Shell \MakeLowercase{\textit{et al.}}: Bare Advanced Demo of IEEEtran.cls for IEEE Computer Society Journals}

\IEEEpubid{\makebox[\columnwidth]{\hfill}} %0000--0000/00/\$00.00~\copyright~2023 IEEE}%
%0000--0000/00/\$00.00~\copyright~2023 IEEE}%
%\hspace{\columnsep}\makebox[\columnwidth]{\hfill}}


\IEEEtitleabstractindextext{%
\begin{abstract}
This paper investigates algorithms for solving distributed consensus optimization problems that are non-convex. 
%As a state-of-art algorithm candidate that solves such problems, a variant of ADMM (Alternating Direction Method of Multipliers) \cite{Boyd2011}, called Consensus ADMM, has been well studied in literature. 
%However, in general, such an algorithm only works for convex problems. Therefore, how to meet distributed consensus optimization problems that are non-convex is still open. 
Since Typical ALADIN (Typical Augmented Lagrangian based Alternating Direction Inexact Newton Method, T-ALADIN for short)\cite{houska2016augmented} is a well-performed algorithm treating distributed optimization problems that are non-convex, directly adopting T-ALADIN to those of consensus is a natural approach. %However, when solving distributed consensus optimization problems, 
However, T-ALADIN typically results in high communication and  computation overhead, which makes such an approach far from efficient. 
In this paper, %to improve T-ALADIN in terms of both communication and computation efficiency,
we propose a new variant of the ALADIN family, coined consensus ALADIN (C-ALADIN for short). C-ALADIN inherits all the good properties of T-ALADIN, such as the local linear or super-linear convergence rate and the local convergence guarantees for non-convex optimization problems; besides, C-ALADIN offers unique improvements in terms of \textit{communication efficiency} and \textit{computational efficiency}. 
%	Specifically, to boost communication efficiency, C-ALADIN adopts \textit{BFGS}\cite{Nocedal2006} in such a novel way that the corresponding Hessian approximation can be performed at the master side, which avoids the transmission of Hessian matrices from agents. Atop the above advancement, C-ALADIN then improves computational efficiency by observing the sparse nature of quadratic programming (QP), which makes C-ALADIN escape from solving a large-scale QP. 
Moreover, C-ALADIN involves a reduced version, in comparison with Consensus ADMM (Alternating Direction Method of Multipliers) \cite{Boyd2011}, showing significant convergence performance, even without the help of second-order information. We also propose a practical version of C-ALADIN, named FedALADIN, that seamlessly serves the emerging federated learning applications, which expands the reach of our proposed C-ALADIN. We provide numerical experiments to demonstrate the effectiveness of C-ALADIN. The results show that C-ALADIN has significant improvements in convergence performance.
\end{abstract}

}

\maketitle

% is not selected <OR> if conference mode is selected - because compsoc
% conference papers position the abstract like regular (non-compsoc)
\IEEEdisplaynontitleabstractindextext

% \fi
\IEEEpeerreviewmaketitle

\ifCLASSOPTIONcompsoc
\IEEEraisesectionheading{\section{Introduction}\label{sec:introduction}}
\else
\section{Introduction}
\label{sec:Introduction}
\fi


\IEEEPARstart{I}n recent years, distributed optimization algorithms have received a lot of attention due to developments in numerical optimal control \cite{diehl2011numerical}, smart grid \cite{zhu2015optimization}, wireless communication\cite{bjornson2017massive}, game theory\cite{han2019game}, and machine learning\cite{mcmahan2017communication}. 
In the field of distributed optimization algorithm design, solving distributed non-convex problems efficiently has always been the direction of people's efforts.
%However, there are still lots of open problems in this field. Such as solving distributed consensus non-convex problems with theoretical guarantees or high convergence rate algorithm design. 
%Although a large number of distributed optimization algorithms have been studied for convex problems, 
% Because many practical problems such as federated learning are usually non-convex.
To deal with non-convexity, in this paper, we follow this direction and propose a novel algorithmic framework for distributed non-convex consensus optimization.% and federated 
\subsection{The Road to Consensus ALADIN}\label{sec: The Road to Consensus ALADIN}

%\subsection{Prior arts}
%Before introducing our proposed method, in Section~\ref{sec: Distributed Consensus Optimization}, we first summarize some existing works on distributed consensus optimization.%the difference between them and federated learning, 
%Then, in Section~\ref{sec:Why T-ALADIN}, we will have an overview of the highly well performed distributed non-convex optimization algorithm named T-ALADIN. Finally, as a potential application of C-ALADIN, we will review the current research state of federated learning in Section~\ref{sec: Federated Learning}.

% \subsubsection{Distributed Consensus Optimization}\label{sec: Distributed Consensus Optimization}
We start with introducing distributed optimization~(DO for short) problems.
DO problems are generally formulated in the fashion of mathematical programming, where
separable objectives are linearly coupled by $m$ equality constraints. Formally, it can be described as follows:
\begin{equation}\label{eq: DOPT_G}
\begin{split}
\min_{\xi_i\in \mathbb R^{n_i}}\;\;&  \mathop{\sum}_{i=1}^{N}  f_i(\xi_i)\\
\mathrm{s.t.}\;\;& \mathop{\sum}_{i=1}^{N} A_i\xi_i=b\;|\mu. 
\end{split}
\end{equation}
Here, the coupling matrices $A_i\in \mathbb R^{n_i\times m}$ and the coupling parameter $b\in \mathbb R^{m}$ are given. The dimension $n_i$ of private variables $\xi_i$s are potentially different. $\mu\in \mathbb R^m$ indicates the corresponding dual variable of the coupling constraints. When the objective $f_i$s are convex, there are some classical algorithms that can be used to solve it, such as dual decomposition (DD)
\cite{dantzig1960decomposition,everett1963generalized}, and ADMM \cite{Boyd2011}.
As a special case of DO, distributed consensus (DC) optimization problems meet all the challenges such as convergence theory for non-convex cases. 
The main difference of DO and DC is that DC has a global variable to which all the private variables will converge (detailed descriptions will be shown in Section \ref{sec: FL via Consensus ADMM}). 

As milestone research of DC, \cite{shi2014linear} shows that Consensus ADMM \cite{Boyd2011}, under some assumptions, has a linear convergence rate for DC problems that are strongly convex. We refer \cite{yang2022survey} as a survey paper for more details. 
Notice that, similar to Consensus ADMM, current algorithms such as DGD \cite{yuan2016convergence}, EXTRA \cite{shi2015extra} only have convergence guarantees for convex problems in the area of DC. However, 
many practical problems \cite{diehl2011numerical}, especially those met in federated learning~(FL)\cite{ioffe2015batch}, are non-convex. When meeting such problems, the above algorithms cannot provide satisfactory solutions in general.

%{\color{red}Why T-ALADIN}

To solve such non-convex problems, in literature, T-ALADIN is a state-of-the-art algorithm that can provide theoretical local convergence guarantee. Technically, T-ALADIN can be regarded as a successful combination of ADMM and sequential quadratic programming (SQP). Details can be found in Subsection \ref{sec: ALADIN chapter}.
To date, T-ALADIN has several elegant successors \cite{Engelmann2019,Du2019,houska2017convex,Shi2022} and already shown effectiveness in many applications \cite{Shi2018,Engelmann2019,Du2019,Jiangwireless}. From the above facts, to solve distributed consensus problems that are non-convex, directly adopting T-ALADIN seems very natural. However, such a trivial approach meets the following challenges: \textbf{\textit{first}}, T-ALADIN will bring a large number of constraints in the coupled QP step (details can be found in \eqref{eq: ALADIN-coupled QP}), and the  dimension of the corresponding dual variable $\mu$ will be extremely large; \textbf{\textit{second}}, the T-ALADIN structure inevitably depends on the uploading of the first and second order information (details can be found in\eqref{eq: ALADIN upload}) from the agents and downloading of the updated primal and dual variables, which incurs huge communication complexity; \textbf{\textit{third}}, in T-ALADIN, a large-scale coupled QP has to be solved exactly, which results in heavy computation workload.

In this paper, we propose C-ALADIN for meeting the aforementioned challenges. C-ALADIN addresses the challenges as follows: \textbf{\textit{First}}, instead of solving a coupled QP in T-ALADIN,  a consensus QP is solved in C-ALADIN; \textbf{\textit{second}}, we improve the upload and download communication efficiency by designing decoding strategies on both sides of the agents and the master.      
In detail, on the uploading side, we find that the local optimizer dominates such parameters. Such an observation, in conjunction with the approximation techniques of Broyden–Fletcher–Goldfarb–Shanno (BFGS) \cite{Nocedal2006}, enable the master to recover the Hessian approximation matrices. It avoids uploading Hessian matrices directly from agents. We name the above techniques \textit{Consensus BFGS ALADIN}.
% helps us avoid  In such a way, the Hessian approximation matrices can be  
% recovered by the master.
Later, in a reduced version named \textit{Reduced Consensus ALADIN}, we simply use an identity matrix for large-scale computation problems.
On the downloading side, inspired by the KKT (Karush–Kuhn–Tucker) optimality condition of the coupled QP, we allow the agents to recover the dual variables that are not urgently broadcast. It can be realized by decoding such variables with the help of the global variable; \textbf{\textit{third}}, in C-ALADIN, the computational bottlenecks come from solving a large-scale sparse consensus QP that plays a key role in coordinating information.
% has to be solved for information coordination in the C-ALADIN framework. 
Inspired by the technique of KKT mentioned above, we find an equivalent form of the large-scale sparse consensus QP in C-ALADIN. Such a form can significantly release the burden on computing the corresponding KKT matrix. Based on our proposed C-ALADIN, we then propose a theory of convergence analysis, which works for both convex and non-convex cases. 

In order to expand the application scope of our proposed C-ALADIN algorithm family, we next show how FL can benefit from our proposed C-ALADIN.

\subsection{Federated Learning via Consensus ALADIN}\label{sec: Federated Learning}

FL, as a framework that aims to train a relatively universal model with data from different devices without transmitting the original data directly, involves many DC problems that are in general either convex or non-convex. In fact, there are several existing efforts \cite{ioffe2015batch,mcmahan2017communication,li2020federated,li2019feddane,zhou2021communication,zhou2022federated}  on solving DC problems in FL. However, such algorithms typically suffer the bottleneck of lacking theoretical non-convex analysis and unsatisfactory convergence rate. 
%typically lack of generalized theoretical non-convex analysis. ht.
%suffer the bottleneck of theoretical non-convex analysis and convergence rate. 
To achieve both, we start adopting C-ALADIN in FL. Though C-ALADIN shows promising convergence performance, directly integrating C-ALADIN with FL is non-trivial. Specifically, all current members in C-ALADIN
rely on uploading the local primal variables to the master, which is not secure. Then, because of the high-dimension property of FL, second order information can not be used, which limits the power of \textit{Consensus BFGS ALADIN}. By observing such challenges, based on \textit{Reduced Consensus ALADIN}, we design a novel variant member of C-ALADIN, named \texttt{FedALADIN}, that seamlessly meets the requirements of FL.

In summary, our key contributions of this paper are as follows:

a) We introduce the notion of consensus QP.

b) We propose a novel efficient algorithm family C-ALADIN that shows rigorous convergence guarantee, which consists of \textit{Consensus BFGS ALADIN} and \textit{Reduced Consensus ALADIN}. 

c) We propose a novel proof framework to perform the convergence analysis of C-ALADIN.

d) We propose \texttt{FedALADIN} for FL.

e) We perform numerical experiments on both C-ALADIN and \texttt{FedALADIN}. The results say that both C-ALADIN and \texttt{FedALADIN} show significant improvements in convergence performance.

\subsection{Organization}
The rest of this paper is organized as follows. 
In Section \ref{sec: section2}, we provide mathematical preliminaries, including fundamentals of Consensus ADMM and T-ALADIN. 
In Section~\ref{sec: C-ALADIN},
we present a new algorithm named C-ALADIN. Moreover, we show that a reduced version of C-ALADIN can be applied to FL named \texttt{FedALADIN} in Section \ref{sec: FedALADIN}.  Later, convergence theory of C-ALADIN is established in Section~\ref{sec: convergence}. In the end, we show the numerical result in Section~\ref{sec: Numerical} .
In Section \ref{related-work}, we provide a literature review. Section \ref{sec: conclusion} concludes this paper.

\section{Preliminaries}\label{sec: section2}

% In this section, we hava a basic overview of Consensus ADMM involved in FL. Then, T-ALADIN is also reviewed.

In this section, we provide formal fundamentals of Consensus ADMM and T-ALADIN.
\subsection{FL via Consensus ADMM}\label{sec: FL via Consensus ADMM}

Assume that we have $N$ clients\footnote{In numerical optimization, one who solves the sub-problem is called agent, and in FL it is called client. In this paper, we use the two notions interchangeably. 
} and each of those has a local dataset $\mathcal D_i$ where $i\in \{ 1, \dots, N\}$. 
Here, the loss function of client $i$ is defined as
\begin{equation}
f_i(z) = \alpha_i\sum_{t_i\in \mathcal{D}_i} l_i(z; t_i), 
\end{equation}
where $l_i(z,t_i): \mathbb R^{n\times|t_i|}\rightarrow \mathbb R$ is a mapping for measuring the prediction error of global variable $z\in \mathbb R^n$.
Moreover $\alpha_i$s are the positive weights with $\sum_{i=1}^{N} \alpha_i =1$. The main goal of machine learning is to solve the following joint optimization problem: 
\begin{equation}\label{eq: FL original}
\min_{z\in \mathbb R^n}	F(z) = \sum_{i=1}^{N} f_i(z).
\end{equation}

However, $\mathcal D_i$s usually belongs to different clients and can not be shared with each other.  
To deal with this, instead of solving Problem~\eqref{eq: FL original}, FL solves the reformulated Problem~\eqref{eq: FL} in a distributed way.
%Naturally, Equation~\eqref{eq: FL original} can be reformulated as the following DC problem
\begin{equation}\label{eq: FL}
\begin{split} 
\min_{x_i,z\in \mathbb R^n}\;\;& \sum_{i=1}^{N}f_i(x_i) \\ \quad\mathrm{s.t.}\;\;& x_i = z \;|\lambda_i. \;
\end{split}
\end{equation}
Here $x_{i}$ denotes the local primal variable of agent $i$ and $z$ indicates the primal global variable. By using   Lagrange multiplier (dual variable) $\lambda_i$,
the corresponding Lagrangian function can be expressed as
\begin{equation}\label{eq: ADMM Lagrangian}
\begin{split}
\mathcal{L}(x_i, z, \lambda_{i})=&\sum_{i=1}^{N}f_i(x_i)
\\
&+\sum_{i=1}^{N}\lambda_{i}^\top (x_i-z) + \sum_{i=1}^{N}\frac{\rho}{2}\|x_i-z\|^2.
\end{split}
\end{equation}
Here, $\rho$ is a given positive penalty parameter. From \eqref{eq: ADMM Lagrangian}, in the \texttt{FedADMM}, the local primal and dual variables can be updated as Algorithm \ref{alg:FedADMM} with learning rate $\eta_i$.
Note that the uploaded information $w_i$s  from client side of Algorithm \ref{alg:FedADMM} is a linear combination of the local primal and the dual variables, which is a secure way of protecting clients' local information.

%\begin{equation}\label{eq : fedadmm}
%	\left\{\begin{split}
%		{x_i}^+&=\mathop{\arg\min}_{x_i} f_i(x_i)+\lambda_{i}^\top x_i +\frac{\rho}{2}  \|x_i -z\|^2, \\
%		\lambda_{i}^+&=\lambda_{i}+\rho({x_i}^+-z),\\
%		z^+&=	\frac{1}{N} \sum_{i=1}^{N}\left(x_i^++ \frac{\lambda_i^+}{\rho}\right).\\
%	\end{split}\right.
%\end{equation}
%Given Equation \eqref{eq : fedadmm}, if the local private variable $x_i$ is updated inexactly, we can get \texttt{FedADMM} provided in Algorithm \ref{alg:FedADMM}.
\begin{algorithm}[H]
\small
\caption{\texttt{FedADMM}: Consensus ADMM for FL}
\textbf{Initialization:} Initial guess  of global model $z=0$, local model $x_i^-=0$s and  dual variables $\lambda_i=0$. Set  Total number of rounds $T$ and  penalty parameter $\rho$.
\vspace{0.1em}	

\textbf{For $t=1\dots T$}
\texttt{Clients:} // In parallel\\
\textbf{ \hspace*{0.5em} For $i=1\dots N$}\\
%	\begin{enumerate}
\hspace*{1em}Download $z$ from the server\\
\hspace*{0.8em}  Locally update $ w_i\leftarrow  \texttt{ClientUpdate} (z,i)$\\
\hspace*{1.1em}Upload $w_i$ to the server\\
%	\end{enumerate}
\textbf{ \hspace*{0.3em} End}

\texttt{Server:}
$
z = \frac{1}{N} \sum_{i=1}^{N} w_i.
$\\
\textbf{End}

\vspace{0.1em}	
$\texttt{ClientUpdate} (z,i)$:\\
\textbf{Input:} Local epoch number $E_i$, client learning rate $\eta_i$.\\
%	\begin{enumerate}
%	\end{enumerate}

\hspace*{0.5em}	\textbf{For $e=1\dots E_i$}

%\begin{enumerate}
\hspace*{1em}	${x_i}=x_i-\eta_i\left(  \nabla f_i(x_i)+\lambda_i^{\text{ADMM}}+\rho\left( x_i-z \right) \right)$
%	\end{enumerate}

\hspace*{0.5em}	\textbf{End}

\hspace*{0.5em}$\lambda_{i}^{\text{ADMM}}=\lambda_{i}^{\text{ADMM}}+\rho({x_i}-z)$.

%	\begin{enumerate}
%	\hspace*{0.5em} $	g_i=\rho(z-x_i^+)-\lambda_i$\\
\vspace{0.5em}
\hspace*{0.5em}	 $w_i=x_i+\frac{1}{\rho} \lambda_{i}^{\text{ADMM}}$
%	\end{enumerate}

\textbf{return:} $w_i$

\label{alg:FedADMM}
\end{algorithm}

\begin{remark}
The private variables $x_i$ can be updated by applying any approximation technologies such as decentralized linearized alternating direction method of multipliers (DLM) \cite{DLM} or decentralized quadratically approximated alternating direction method of multipliers (DQM) \cite{DQM}. 
\end{remark}

\begin{remark}
As pointed out in \cite[Section IV]{yang2022survey}, a Consensus ADMM variation \cite{shi2014linear} was specifically designed  for solving Problem \eqref{eq: FL} instead of Problem \eqref{eq: DOPT_G} and should be considered a relatively independent algorithm. 
\end{remark}

%	\begin{algorithm}[H]
%		%	\small
%		\caption{\texttt{FedADMM}: Consensus ADMM for FL}
%		%\textbf{Initialization:} Initial guess $(p^0,\lambda^0,\mu^0)$, choose $\Sigma_i,\rho^0,\mu^0,\epsilon$. \\
%		\textbf{Repeat:}
%		\begin{enumerate}
%			\item Each client learn their parameter $(x_i^+, \lambda_{i}^+)$ locallly and uplaod it to the server:
%			\begin{equation}
%				\left\{\begin{split}
%					{x_i}^+&=\mathop{\arg\min}_{x_i} f_i(x_i)+\lambda_{i}^\top x_i +\frac{\rho}{2}  \|x_i -z\|^2 \\
%					\lambda_{i}^+&=\lambda_{i}+\rho({x_i}^+-z).
%				\end{split}\right.
%			\end{equation}
%		Compute $w_i = x_i^++\frac{1}{\rho} \lambda_{i}^+$ and transmit it to the server.
%			\item The server collects updated local variables ${x_i}^+$ from each clients and do average aggregate sum 
%			\begin{equation}
%				\begin{split}
%					z^*&= \mathop{\arg\min}_{z} \sum_{i=1}^{N}\left( \frac{\rho}{2} \|x_i^+-z\|^2-(\lambda_{i}^+)^\top z\right)\\
%					&=\frac{1}{N} \sum_{i=1}^{N}x_i^+ +\frac{1}{N\rho}\sum_{i=1}^{N}\lambda_{i}^+\\
%					&=\frac{1}{N} \sum_{i=1}^{N} w_i.
%				\end{split}
%			\end{equation}
%			and broadcast the global model $z\leftarrow z^*$.
%		\end{enumerate}
%	%	\label{alg:FedADMM}
%	\end{algorithm}
%%

%Note that, the (stochastic) gradient descent method is used in the local primal update. This does not guarantee that the subproblem is solved accurately. This is a major difference between FL and distributed optimization, and results in the fact that most optimization algorithms cannot be applied directly for FL.

\subsection{T-ALADIN in A Nutshell}\label{sec: ALADIN chapter}
T-ALADIN is the first distributed optimization algorithm that can generally solve non-convex DO \eqref{eq: DOPT_G}.

In the first step, Algorithm~\ref{alg:ALADIN} has a similar operation as ADMM and gets new local optimizers from each client. Later we evaluate the gradients $g_i$s and positive definite Hessian matrices approximation $H_i$s with $\xi_i^+$s. In the third step, we solve a large-scale QP for coordination by using $g_i$s and $H_i$s from the second step. In the end, we send the updated primal and dual variables to each client in the forth step.

The main difference between ADMM and T-ALADIN is that the latter updates the global dual variable $\mu$ by solving the constrained coupled QP (Equation \eqref{eq: ALADIN-coupled QP}). %by using Hessian and gradient from each subproblems in the third step of Algorithm~\ref{alg:ALADIN}.  
Moreover, with the assumption of linearly independent constraint qualification (LICQ) and second order sufficient condition (SOSC), Algorithm~\ref{alg:ALADIN} has local convergence guarantees for distributed non-convex problems \cite{Houska2016}.

Note that T-ALADIN tries to collect all the good properties of SQP and ADMM. There are two main extreme cases:
\begin{itemize}
\item When $\rho\rightarrow \infty$, the first step of Algorithm~\ref{alg:ALADIN} is redundant
and the whole structure is equivalent to SQP.
\item When $\rho\rightarrow 0$, Algorithm~\ref{alg:ALADIN} equals a combination of DD and Newton's method.
%	\item When $\infty >\rho> 0$ and $H_i=A_i^\top A_i$, ALADIN  
%	equals to a variant of ADMM.
\end{itemize}

\begin{algorithm}[H]
\small
\caption{Typical ALADIN}
\textbf{Initialization:} Initial guess of primal and dual variables $(y_i,\mu)$.  \\
\textbf{Repeat:}
\begin{enumerate}
\item Parallelly solve local nonlinear programming (NLP) without subsystem coupling:
\begin{equation}\label{ALADIN-step1}
{\xi_i}^+=\mathop{\arg\min}_{\xi_i} f_i(\xi_i)+\mu^\top A_i \xi_i+\frac{\rho}{2}\|\xi_i-y_i\|^2.
\end{equation}
\item Evaluate Hessian approximation and gradient from ${\xi_i}^+$:
\begin{equation}\label{eq: ALADIN upload}
\left\{
\begin{array}{l}
\begin{split}
B_{i}\approx &\nabla^2 f_i(\xi_i^+)\succ 0, \\
g_{i}=&\nabla f_i(\xi_i^+).
\end{split}
\end{array}
\right.
\end{equation}
\item Solve the coupled QP on master side:
\begin{equation}\label{eq: ALADIN-coupled QP}
\begin{split}
\mathop{\mathrm{\min}}_{ \Delta \xi_i}& \quad \mathop{\sum}_{i=1}^{N} \frac{1}{2}\Delta \xi_i^\top B_{i} \Delta \xi_i+g_{i}^\top \Delta \xi_i \\
\mathrm{s.t.} &\mathop{\sum}_{i=1}^{N} A_i(\xi_i^++\Delta \xi_i)=b\; |{\mu}^+	.
\end{split}
\end{equation}

\item Download :
\begin{equation}
\left\{
\begin{array}{l}
\begin{split}
\mu& \leftarrow {\mu}^+\\
y_{i}&\leftarrow \xi_i^++\Delta \xi_i
\end{split}
\end{array}
\right.
\end{equation} 
\end{enumerate}
\label{alg:ALADIN}
\end{algorithm}

\begin{remark}
To the best of our knowledge, although a convex LASSO problem has been solved with T-ALADIN \cite[Section 5.1]{houska2017convex}, is a two objectives optimization problem. Besides, no such concept like \emph{consensus ALADIN} has been formally proposed by others.
\end{remark}

\section{Consensus ALADIN}\label{sec: C-ALADIN}
% \st{T-ALADIN is the first second-order algorithm in the area of distributed optimization. Whether second-order (approximate) information is used is not a classification of an algorithm in distributed optimization. In fact, it is that the increasement of the primal variables are optimized together when the dual variable is updated. The special operation in ALADIN can speed up convergence to satisfy the coupling constraints. When the subproblems are unconstrained, Algorithm~\ref{alg:ALADIN} can converge to a feasible solution with only one iteration. }

%In this section, we will show our methods for solving the three challenges mentioned Section \ref{sec:}. %Moreover, different from Algorithm \ref{alg:ALADIN}, the proposed algorithm is mainly designed for solving Problem \eqref{eq: FL}.

In Subsection \ref{sec: ALADIN structure},  we propose the consensus QP.% to simplify the equality constraints.
In Subsection~\ref{sec: Derivative-Free Consensus ALADIN}, we detail our communication-efficient design.  
%we decode the corresponding local Hessian and gradient with BFGS method for improving uploading efficiency. On the downloading side, modified with KKT expression. 
%inspired by \cite{houska2017convex}, we point out that C-ALADIN can still survive without transmitting the corresponding Hessian and gradient \eqref{eq: ALADIN upload}. 
In Subsection~\ref{sec: Sparse QP Solver}, we provide our techniques that improve computational efficiency. In Subsection \ref{sec: conALADIN}, with the techniques proposed in the above subsections, we formally describe C-ALADIN with two variants, namely Consensus BFGS ALADIN and Reduced Consensus ALADIN.

% Subsequently, we analyze the difference between Consensus ADMM and Reduced Consensus ALADIN in Subsection \ref{sec: Comparison}.

%	\newpage

\subsection{From Coupled QP to Consensus QP}\label{sec: ALADIN structure}
Remember that in Subsection \ref{sec: The Road to Consensus ALADIN}, T-ALADIN is struggling with solving Problem \eqref{eq: FL} because of the coupling equality constraints in  Equation \eqref{eq: ALADIN-coupled QP}. Thus, the first step towards making T-ALADIN survive in  \eqref{eq: FL} is to
reconstruct Equation \eqref{eq: ALADIN-coupled QP}, which yields our consensus QP that is formally shown as follows: 
% With the reconstruction of the coupling equality constraints, Equation \eqref{eq: ALADIN-coupled QP} can be changed into \eqref{eq: consensus QP} in the consensus sense.
\begin{equation}\label{eq: consensus QP}
\begin{split}
\mathop{\mathrm{\min}}_{ \Delta x_i, z}& \quad \mathop{\sum}_{i=1}^{N} \left(\frac{1}{2}\Delta x_i^\top B_i \Delta x_i+g_{i}^\top \Delta x_i\right) \\
\mathrm{s.t.} &\qquad\Delta x_i+x_i^+=z\; |\lambda_{i}.
\end{split}
\end{equation}
It is trivial to see that, in \eqref{eq: consensus QP}, by introducing a global variable $z$ and coupling all $x_i$s to $z$. The number of coupled equality constraints will be reduced to $O(N)$ which is lower than $O(N^2)$ that comes from the direct adoption of T-ALADIN (shown as Figure \ref{fig: dense} and \ref{fig: sparse}). Note that,  
% the only difference, between  and , 
in Equation \eqref{eq: consensus QP}, we have one global primal variable $z$ and $N$ dual variable, which is in contrast to the formulation of Equation \eqref{eq: ALADIN-coupled QP}.

%one global dual variable $\mu$ and $N$ private primal variables $\xi_i$s are consist in the T-ALADIN fashion, while C-ALADIN is the opposite.
% When T-ALADIN is applied directly to distributed consensus optimization, the number of corresponding equality constraint will be of order of $O(N^2)$.  
\begin{figure}[htbp]
\centering
\begin{minipage}{0.49\linewidth}
\centering
\includegraphics[width=1\linewidth]{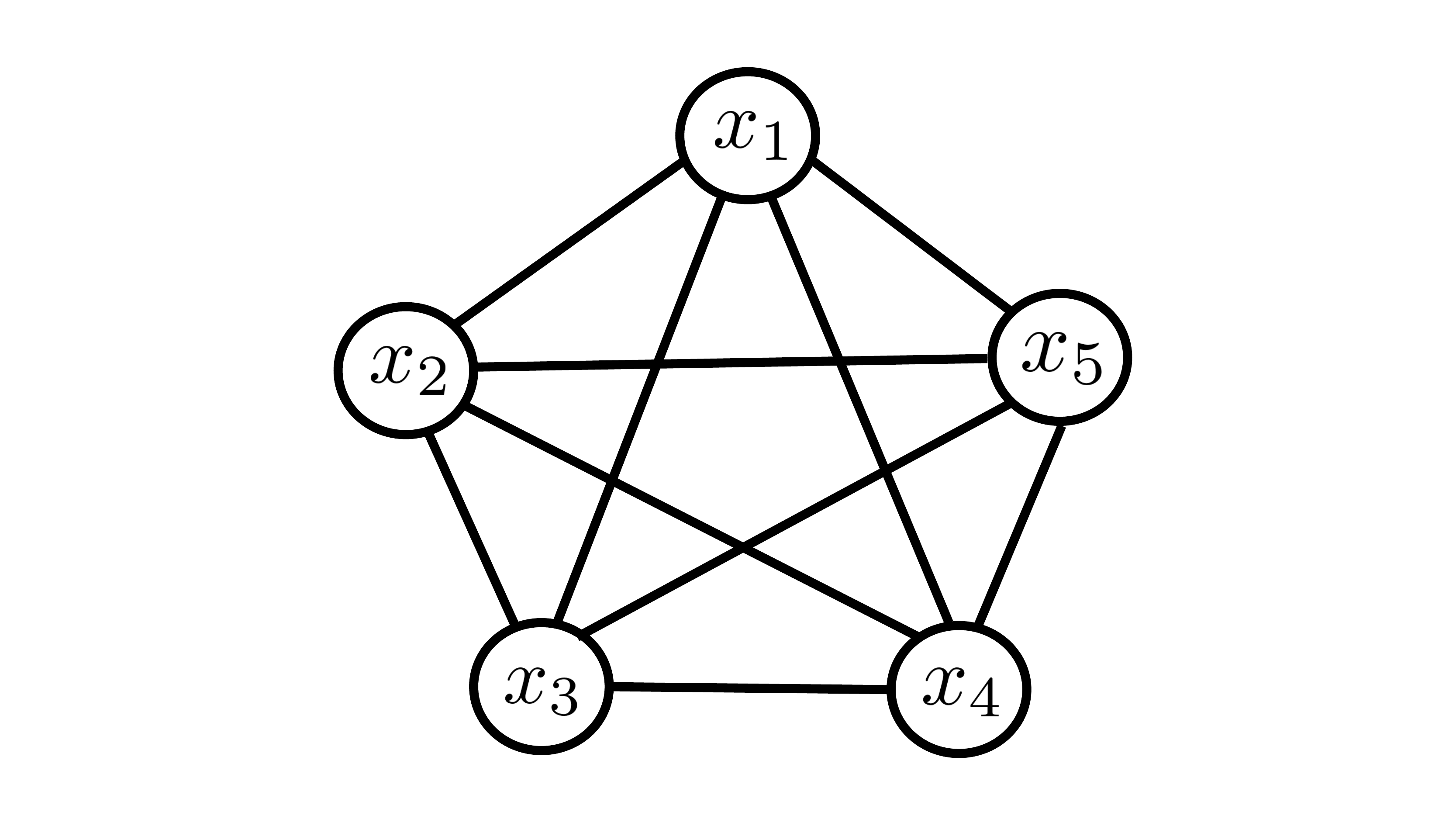}
\caption{Fully Connection \\(Worst case scenario).}
%	\caption{Linear regression with different $\rho$ (convex problem).}
\label{fig: dense}%文中引用该图片代号
\end{minipage}
\begin{minipage}{0.49\linewidth}
\centering
\includegraphics[width=1\linewidth]{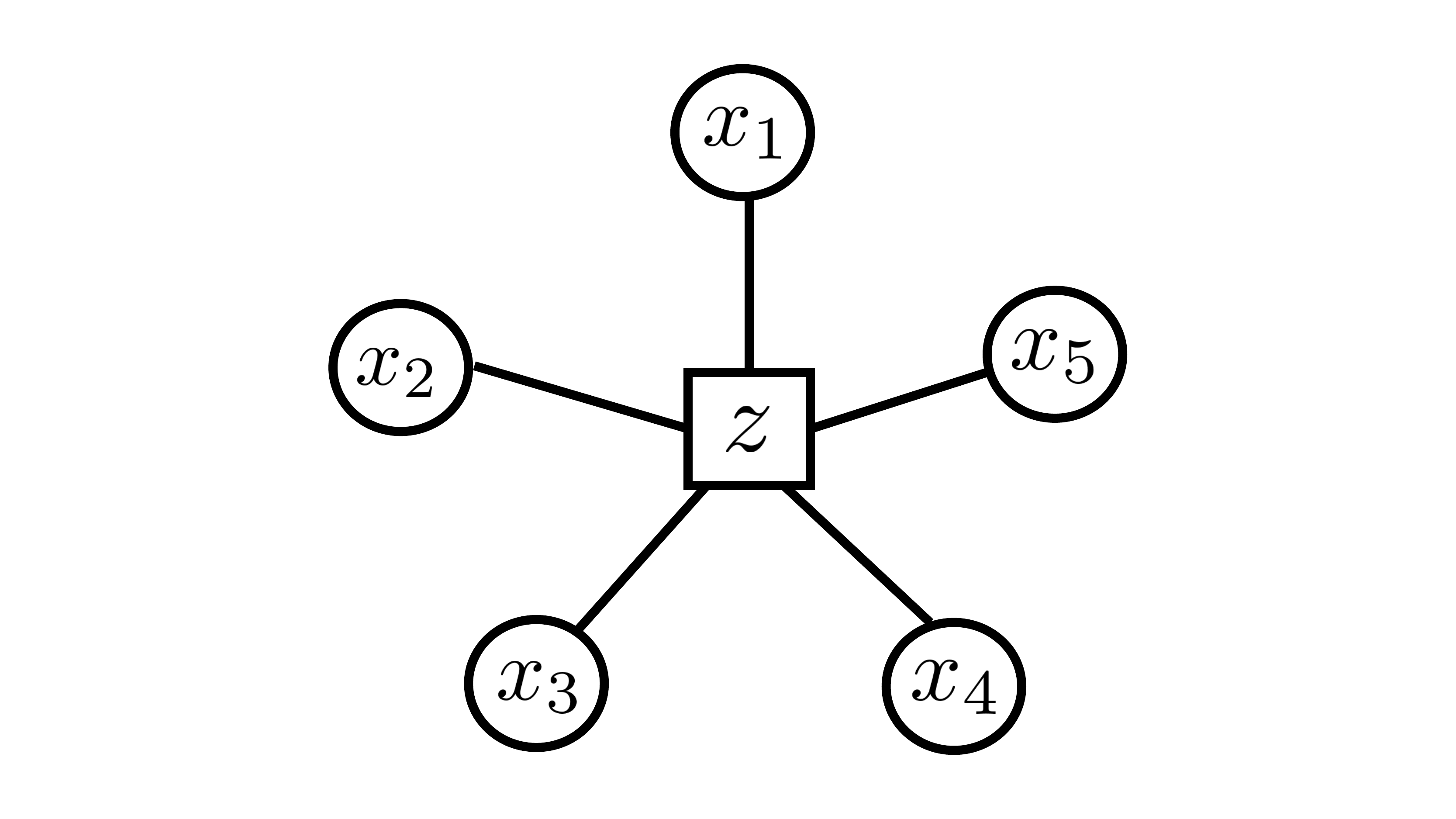}
\caption{Private-master coupling.}
\label{fig: sparse}%文中引用该图片代号
\end{minipage}
\end{figure}

%; b) the constraints of consensus QP \eqref{eq: ALADIN-coupled QP} and \eqref{eq: consensus QP} in the two frameworks are different since the global primal variable $z$.
%
After replacing Equation \eqref{eq: ALADIN-coupled QP} with Equation \eqref{eq: consensus QP}, Algorithm \ref{alg:ALADIN} starts working in a consensus fashion. We name this C-ALADIN.

% We call this algorithm, based on Equation \eqref{eq: consensus QP}, as C-ALADIN. It can be treated as a new member of the ALADIN family.

\begin{remark}\label{remark: T-ALADIN}
Regarding the number of coupling constraints of T-ALADIN, we want to stress that $O(N^2)$, corresponding to the Figure \ref{fig: dense}, describes the worst case when directly adopting T-ALADIN in solving DC problems. As for the best case, though T-ALADIN can work under $O(N)$ coupling constraints as those of C-ALADIN, it needs a fine-grained design on $A_i$, which, from user's perspective, hinders the ease of practical adoption of T-ALADIN.
\end{remark}

%\begin{equation}\label{eq: coupling}
%	\begin{split}
%	A=	\left\{
%		\underbrace{\begin{bmatrix}
%				\mathbf{I}\\
%				\mathbf{I}\\
%				\mathbf{I}\\
%				\mathbf{I}\\
%				\mathbf{0}\\
%				\mathbf{0}\\
%				\mathbf{0}\\
%				\mathbf{0}
%		\end{bmatrix}}_{A_1} \underbrace{\begin{bmatrix}
%	\mathbf{-I}\\
%	\mathbf{0}\\
%	\mathbf{0}\\
%	\mathbf{0}\\
%	\mathbf{I}\\
%		\mathbf{I}\\
%	\mathbf{I}\\
%	\mathbf{0}
%	\end{bmatrix}}_{A_2}
%\underbrace{\begin{bmatrix}
%		\mathbf{0}\\
%		\mathbf{-I}\\
%		\mathbf{0}\\
%		\mathbf{0}\\
%		\mathbf{0}\\
%			\mathbf{0}\\
%		\mathbf{0}\\
%		\mathbf{0}
%\end{bmatrix}}_{A_3}
%\underbrace{\begin{bmatrix}
%		\mathbf{0}\\
%		\mathbf{0}\\
%		\mathbf{-I}\\
%		\mathbf{0}\\
%		\mathbf{0}\\
%			\mathbf{0}\\
%		\mathbf{0}\\
%		\mathbf{0}
%\end{bmatrix}}_{A_4}
%\underbrace{\begin{bmatrix}
		%		\mathbf{0}\\
		%		\mathbf{0}\\
		%		\mathbf{0}\\
		%		\mathbf{-I}\\
		%		\mathbf{0}\\
		%			\mathbf{0}\\
		%		\mathbf{0}\\
		%		\mathbf{0}
		%\end{bmatrix}}_{A_5}\right\}
		%	\end{split}
	%\end{equation}

	%\begin{figure}[H]
	%	\centering
	%	\includegraphics[width=0.48\textwidth,height=0.2\textheight]{Img/ALADIN-FAMILY}
	%	\caption{The ALADIN family.}
	%	\label{fig: ALADIN-FAMILY}
	%\end{figure}
	
	\subsection{Improve the Communication Efficiency of C-ALADIN}\label{sec: Derivative-Free Consensus ALADIN}
	In this subsection, we jointly improve the upload \ref{sec: upload} and download \ref{sec: download} communication efficiency of C-ALADIN.

	% the communication efficiency design of C-ALADIN is divided into two parts: the upload \ref{sec: upload} and download \ref{sec: download}.

	%  we will show an efficient way of
	%			the upload \ref{sec: upload} and download in C-ALADIN \ref{sec: download}.
	
	\subsubsection{Improving Upload Communication Efficiency}\label{sec: upload}
	In the C-ALADIN framework, instead of solving Problem \eqref{ALADIN-step1}, we solve the following decoupled augmented loss function \eqref{eq: consensus ALADIN-step1} with local dual $\lambda_i$.
	\begin{equation}\label{eq: consensus ALADIN-step1}
		{x_i}^+=\mathop{\arg\min}_{x_i} f_i(x_i)+\lambda_i^\top  x_i+\frac{\rho}{2}\|x_i-z\|^2.
	\end{equation}
	
	In order to avoid uploading the gradient and Hessian approximation directly \eqref{eq: ALADIN upload}, we choose to decode the first and second order information on the master side, which is detailed in \eqref{eq: BFGS}
	\begin{equation}\label{eq: BFGS}
		\left\{
		\begin{split}
			&g_i(x_i^+)=\;\rho(z-x_i^+)-\lambda_i\quad\text{((sub)gradient)},\\
			&s_i(x_i^+,x_i^-)=\;x_i^+-x_i^-,\\
			&y_i(x_i^+,x_i^-)=\; g_i(x_i^+)-g_i^-,\\
			&B_i^{+}=\;B_i-\frac{B_i s_i s_i^\top B_i}{s_i^\top B_i s_i}+\frac{y_iy_i^\top}{s_i^\top y_i}\quad\text{(BFGS update)}.
		\end{split}
		\right.
	\end{equation}
	We want to stress several key designs behind Equation \eqref{eq: BFGS}. First, we suppose that the augmented NLP \eqref{eq: consensus ALADIN-step1} can be solved exactly. By applying Clarke sub-differential of $f_i$ at $x_i^+$, we have $g_i=\rho(z-x_i^+)-\lambda_i\in \partial f_i$.  It means that the (sub)gradient can be decoded with $x_i^+$ without being transmitted.
	
	Second, by applying the difference of local private variables and the local (sub)gradients, the BFGS Hessian approximation can also be decoded by the master. In order to ensure the positive definiteness of local BFGS matrices, we adopted the strategy of \emph{damped BFGS} \cite[Page 537]{Nocedal2006}, that is, we modify the local gradient difference $y_i$ 
	\[y_i=y_i + \theta (B_is_i-y_i)\]with a turning parameter 
	\[\theta= \frac{0.2(s_i)^\top B_i s_i-(s_i)^\top y_i}{(s_i)^\top B_i s_i-(s_i)^\top y_i}\]
	if 
	\[(y_i)^\top s_i\leq\frac{1}{5} (s_i)^\top B_i s_i.\]
	Note that if $\theta =1 $, $B_i^{+}=B_i$.  The damping thus ensures that the positive curvature of the Hessian in direction $s_i$, which is expressed in the term $\left((s_i)^\top B_i s_i\right)$, will never decrease
	by more than a factor of 5. Modified versions of BFGS may also work.For example, to solve the storage problem of the BFGS approximation Hessian, \emph{limited memory BFGS} (L-BFGS) maybe a promising solution.
	
	%Note that the local convergence can be guaranteed for two assumptions:
	%\begin{itemize}
	%	\item $B_i$ satisfies $\nabla^2 f_i(x_i^+)+ \rho B_i\succ 0$, and this indicates $x_i^+$ is  at least a strickly local minimizer of the decoupled augmented problems.
	%	\item Local linear or superlinear convergence can be guaranteed if $\left\| \nabla^2f_i(x_i)-B_i \right\|$ is sufficiently small.
	%\end{itemize}
	\begin{remark}
		If we are dealing with a single objective problem, instead of computing the inverse of $B_i^+$ directly, we have the following closed form:
		\begin{equation*}
			\begin{split}
				(B^{k+1})^{-1}= &(B^k)^{-1}+\frac{(s^\top y+y^\top (B^{k})^{-1}y)(ss^\top)}{(s^\top y)^2}\\ &-\frac{(B^k)^{-1}y s^\top +s y^\top (B^k)^{-1}}{s^\top y}.
			\end{split}
		\end{equation*}It can not be applied directly if we have a summation of $B_i$s.
	\end{remark}

	\subsubsection{Improving Download Communication Efficiency}\label{sec: download}
	If we express the Lagrange function of Equation~\eqref{eq: consensus QP},
	\begin{equation}%\label{eq: KKT}
		\begin{split}
			\mathcal{L}^{\text{QP}}(\Delta x_i, z, \lambda_{i})=&\left(\sum_{i=1}^{N}\frac{1}{2}\Delta x_i^\top B_i \Delta x_i+g_i^\top \Delta x_i\right)\\
			&+ \left(\sum_{i=1}^{N}\lambda_{i}^\top (\Delta x_i+x_i^+-z)\right),
		\end{split}
	\end{equation}
	the KKT system can be then expressed in the following three equations,
	\begin{equation}\label{eq: KKT}
		\left\{
		\begin{split}
			\frac{\partial \mathcal L^{\text{QP}}}{\partial \Delta x_i}&=B_i\Delta x_i+g_i+\lambda_{i} =0,\\
			\frac{\partial \mathcal L^{\text{QP}}}{\partial \lambda_{i}}&= \Delta x_i+x_i^+-z=0,\\
			\frac{\partial \mathcal L^{\text{QP}}}{\partial z}&= -\sum_{i=1}^{N}\lambda_{i}=0,
		\end{split}\right.
	\end{equation}
	which implies $\Delta x_i=z-x_i^+$ and $\lambda_{i}=B_i(x_i^+-z)-g_i$. It shows that  agents can decode the dual variable $\lambda_i$ with the global variable $z$ without being transmitted.

	%With the above discussion, we can simplify the transmission data and the new \emph{derivative-free consensus ALADIN} algorithm can be summarized as Algorithm \ref{alg:Derivative-Free consensus ALADIN}.

	In summary, in C-ALADIN, the agents only upload their private variables update $x_i^+$ to the master while the master broadcast the global variable (aggregate model) $z$. In this way, neither Hessian nor gradient needs to be uploaded and the dual variables need not to be downloaded.

	%Traditional federated learning algorithms like \texttt{Fedavg}, \texttt{FedSGD} 
	%or \texttt{FedProx} only have a average operation on the server side, however it does not bring out the full computation potential of the server.
	%It is clear that when $N=2$, Algorithm \ref{alg:Derivative-Free consensus ALADIN} almost  boils down to a special case of Algorithm \ref{alg:ALADIN}. However,  since the existence of global variable $z$, when $N\neq 2$, Algorithm \ref{alg:Derivative-Free consensus ALADIN} and Algorithm \ref{alg:ALADIN} are not mutually dependent. The global variable $z$ is an auxiliary variable without its own objective. On the other hand, in the framework of consensus ALADIN, multiple dual variables are need while only one is needed in standard ALADIN. The difference the two algorithms  performs a symmetric statement between the primal and the dual.
	%It can also be seen from the Section \eqref{sec: convergence}, although our global convergence analysis inherits the idea of global convergence proof of \cite{Houska2021}  with convex problem,  the corresponding Lyapunov functions are not the same.
	
	Improving the communication efficiency mentioned above is not the end. To further enhance our algorithm, we next boost computational efficiency.

	\subsection{Improve the Computational Efficiency of C-ALADIN}\label{sec: Sparse QP Solver}
	
	To improve the computational efficiency, a trivial approach is to seek help from existing QP solvers which is described as follows:
	% For the coordinate step \eqref{eq: consensus QP}, 
	% large scale sparse QP solver can be prepared. 
	solvers based on active set \texttt{qpOASES} \cite{ferreau2014qpoases}, \texttt{MOSEK} \cite{aps2019mosek}, \texttt{GUROBI} \cite{gurobi2018gurobi}; solvers based on interior point methods \texttt{CVXGEN} \cite{mattingley2012cvxgen}  and \texttt{OOQP} \cite{gertz2003object}; solvers basd on ADMM or operator splitting method  \texttt{OSQP} \cite{osqp}.
	
	However, such solvers ignore the special structure of our input \eqref{eq: consensus QP}. Therefore, a customized QP solver is needed for improving the computational efficiency.
	%	\subsection{Dense CLose From of the Global Parameter}	
	Next, we show our technical details. 
	
	The global primal variable can be updated in a cheap operation by introducing the following theorem.
	\begin{theorem}
		With the decoding of Hessian approximation and the gradients of the agents by Equation \eqref{eq: BFGS}, the global $z$ can be updated as Equation \eqref{eq: reduced QP}  \footnote{Assume that the matrix inverse operation can be applied on the master side, instead of an inverse operation, \emph{conjugate gradient descent} can be also applied here. }\footnote{ Suppose that for iteration $k$ with $\text{log}_K(k)\in \mathbb{N} $ \cite{Shi2022}, the Hessian matrices can be chosen optionally to upload from the agent to the master. The decomposition of a positive definite matrix can be prepared before the consensus QP step \eqref{eq: consensus QP}. }.
		\begin{equation}\label{eq: reduced QP}
			\begin{split}
				z^{+}_{\text{BFGS}}=\left(\sum_{i=1}^N  B_i^{+} \right)^{-1}\left( \left( \sum_{i=1}^N B_i^{+}x_i^+\right) -  \left( \sum_{i=1}^N g_i  \right) \right). 
				%	=&\frac{1}{N} \sum_{i=1}^{N} \left(x_i^+- \frac{g_i}{\rho}\right).
			\end{split}
		\end{equation}
	\end{theorem}

	\begin{proof}
		
		Linear system \eqref{eq: KKT}  can be expressed as the following  dense form:
		\begin{equation}\label{eq: dense form}
			\underbrace{\begin{bmatrix}
					%\begin{array}{c|c|c}
					\mathcal B&\mathbf I&\mathbf 0\\[0.5mm]
					%	\hline
					\mathbf I&\mathbf 0&-\mathcal I\\[0.5mm]
					%	\hline
					\mathbf 0^\top&(-\mathcal I)^\top& \mathbf 0\\[0.5mm]
					%\end{array}
			\end{bmatrix}}_{M_{\text{KKT}}\in \mathbb R_+^{|(2N+1)n|\times|(2N+1)n|}}	
			\begin{bmatrix}
				\Delta x\\
				%\hline
				\mathbf{\lambda}\\
				%\hline
				z
			\end{bmatrix}=
			\begin{bmatrix}
				-\mathcal G\\
				%\hline
				-x^+\\
				%\hline
				\mathbf 0
			\end{bmatrix}
		\end{equation}
		with 
		\begin{equation*}
			\mathcal B=
			\begin{bmatrix}
				B_1&0&\cdots&0\\
				0&B_2&\cdots&0\\
				\vdots&\vdots&\ddots&\vdots\\
				0&0&\cdots&B_N
			\end{bmatrix},\quad
			\mathbf I= \begin{bmatrix}
				I&0&\cdots&0\\
				0&I&\cdots&0\\
				\vdots&\vdots&\ddots&0\\
				0&\cdots&\cdots&I
			\end{bmatrix},
		\end{equation*}
		\begin{equation*}
			\mathcal I=\begin{bmatrix}
				I&I&I&I
			\end{bmatrix}^\top,
		\end{equation*}
		\begin{equation*}
			\begin{bmatrix}
				\Delta x\\
				%\hline
				\lambda\\
				%\hlin
				z
			\end{bmatrix}=
			\begin{bmatrix}
				\begin{array}{c}
					\Delta x_1\\
					\Delta x_2\\
					\vdots\\
					\Delta x_N\\
					\hline
					\lambda_{1}\\
					\lambda_{2}\\
					\vdots\\
					\lambda_{N}\\
					\hline
					z
				\end{array}
			\end{bmatrix}\quad\text{and} \quad	\begin{bmatrix}
				-\mathcal G\\
				%\hline
				-x^+\\
				%\hline
				\mathbf 0
			\end{bmatrix} =  
			\begin{bmatrix}
				\begin{array}{c}
					-g_1\\
					-g_2\\
					\vdots\\
					-g_N\\
					\hline
					-x_1^+\\
					-x_2^+\\
					\vdots\\
					-x_N^+\\
					\hline
					\mathbf	0
				\end{array}
			\end{bmatrix}.
		\end{equation*}
		Directly solving QP \eqref{eq: consensus QP} is equivalent to solving linear system \eqref{eq: dense form}.  We note that the roadblock to making C-ALADIN computationally efficient is the heavy computational workload,  incurred by deriving the inverse of the large-scale matrix $M_{\text{KKT}}$.   
		
		% C-ALADIN burdens with the inverse of the large-scale matrix $M_{\text{KKT}}$.   

		Inspired by \emph{Schur Complement} we have the following proof. % we find a close form of $z$ update.
		From the first row of Equation~\eqref{eq: dense form}, we have
		\begin{equation}\label{eq: Deltx}
			\Delta x=-\mathcal B^{-1}(\lambda+\mathcal G).
		\end{equation}
		Then plug \eqref{eq: Deltx} into the second row of Equation \eqref{eq: dense form} to get Equation \eqref{eq: Deltx2}.
		\begin{equation}\label{eq: Deltx2}
			\mathcal B^{-1}(\lambda+\mathcal G)+\mathcal I z = x^+\Rightarrow \lambda = \mathcal B(x^+-\mathcal Iz) - \mathcal G. 
		\end{equation}
		Next,  bring Equation \eqref{eq: Deltx2} into the third row of Equation \eqref{eq: dense form}, the following equation will be obtained
		\begin{equation}\label{sec: mix}
			(-\mathcal I)^\top\mathcal B(x^+-\mathcal Iz) +(\mathcal I)^\top \mathcal G=0.
		\end{equation}
		The update of global variable $z$ can be then expressed as follows from Equation \eqref{sec: mix}.
		\begin{equation*}
			\begin{split}
				z^{+}_{\text{BFGS}}=&\underbrace{\left(\mathcal I^\top \mathcal B \mathcal I\right)^{-1}}_{\mathcal K\in \mathbb R_+^{|n|\times |n|}}\left(\mathcal I^\top \mathcal B x^+-(\mathcal I)^\top\mathcal G\right)\\
				=&\left(\sum_{i=1}^N  B_i^{+} \right)^{-1}\left( \left( \sum_{i=1}^N B_i^{+}x_i^+\right) -  \left( \sum_{i=1}^N g_i  \right) \right). 
				%	=&\frac{1}{N} \sum_{i=1}^{N} \left(x_i^+- \frac{g_i}{\rho}\right).
			\end{split}
		\end{equation*}
	\end{proof}
	
	%\begin{equation}\label{eq: reduced QP}
	%	\begin{split}
		%		z^{+}=&\underbrace{\left(\mathcal I^\top \mathcal R \mathcal I\right)^{-1}}_{\mathcal K\in \mathbb R_+^{|m|\times |m|}}\left(\mathcal I^\top \mathcal R x^+-(\mathcal I)^\top\mathcal G\right)\\
		%		=&\left(\sum_{i=1}^N  \rho \right)^{-1}\left( \left( \sum_{i=1}^N \rho x_i^+\right) -  \left( \sum_{i=1}^N g_i  \right) \right)\\
		%		=&\frac{1}{N} \sum_{i=1}^{N} \left(x_i^+- \frac{g_i}{\rho}\right).
		%	\end{split}
	%\end{equation}
	%Or we can raplace it by using (decentralized) conjugate gradient descent.

	As an extension, if $B_i$s are set as $\rho I$, Equation \eqref{eq: reduced QP} 
	is reduced as 
	\begin{equation}\label{eq: reduced QP2}
		z^+_{\rho}=\frac{1}{N} \sum_{i=1}^{N} \left(x_i^+- \frac{g_i}{\rho}\right). 
	\end{equation}
	Here, the expression of \eqref{eq: reduced QP2} has almost the computation complexity as Consensus ADMM since no inverse of aggregated Hessian matrix is needed. 
	%		the global variable $z$ can be updated in the way of equation \eqref{eq: reduced QP} by using the corresponding BFGS Hessian approximation $B_i^+$. 
	Both methods of  
	updating the global variable (\eqref{eq: reduced QP} and \eqref{eq: reduced QP2}) avoid computing $\Delta x_i$ and $\lambda_i$ together.
	In this way, the operation burden of large-scale matrix inversion in Equation \eqref{eq: dense form} is reduced. 
	\begin{remark}
		With Equation \eqref{eq: reduced QP2}, the primal increment $\Delta x_i$ and dual $\lambda_i$ can be then decoded by the agents with
		\begin{equation}\label{eq: decode}
			\left\{
			\begin{split}
				\Delta x_i&=z-x_i^+,\\
				\lambda_{i}&=\rho(x_i^+-z)-g_i
			\end{split}\right.
		\end{equation}
		in the next iteration. 
	\end{remark}
	
	%\begin{equation}
	%	\begin{split}
		%		z^{+}_{\text{BFGS}}=\left(\sum_{i=1}^N  B_i^{+} \right)^{-1}\left( \left( \sum_{i=1}^N B_i^{+}x_i^+\right) -  \left( \sum_{i=1}^N g_i  \right) \right). 
		%	\end{split}
	%\end{equation}
	By combining  the technologies proposed in the above three subsections, 
	we illustrate two variations of C-ALADIN in the next subsection.

	\subsection{Algorithm Structure}\label{sec: conALADIN}
	In this subsection, by combining our proposed techniques described in Subsection \ref{sec: ALADIN structure}-\ref{sec: Sparse QP Solver},  we propose two algorithms. Specifically, one, named \emph{Consensus BFGS ALADIN}, benefits from the techniques of BFGS Hessian approximation (with Equation \eqref{eq: reduced QP}). The other, named \emph{Reduced Consensus ALADIN} (with Equation \eqref{eq: reduced QP2}),  under the scenario where the convergence rate is degraded, can work without the second-order information. We detail Consensus BFGS ALADIN in Algorithm \ref{alg:BFGS ALADIN2}. For \emph{Reduced Consensus ALADIN}, it can be easily got from  by replacing $x_i^{+}$, $\lambda_i$, $g_i$ and $z$ in Algorithm \ref{alg:BFGS ALADIN2} with the ones defined by the following equation:
	% and main step of  Reduced Consensus ALADIN in Equation \eqref{eq: RCA}.
	\begin{equation}\label{eq: RCA}
		\left\{ \begin{split}
			&\lambda_i=\rho(x_i-z)-g_i,\\
			&{x_i}^+=\mathop{\arg\min}_{x_i}f_i(x_i)+\lambda_i^\top  x_i+\frac{\rho}{2}\|x_i-z\|^2,\\
			&g_i=\rho(z-x_i^+)-\lambda_i,\\
			&z=	\frac{1}{N} \sum_{i=1}^{N}\left(x_i^+- \frac{g_i}{\rho}\right).
		\end{split}\right.
	\end{equation}
	Note that, both two algorithms belong to the class of C-ALADIN. % While \emph{derivative-free} here is a setting that Equation \eqref{eq: subgradient} always holds if the sub-problems can be exactly solved. 
	\begin{remark}
		As a supplement, we analyze the difference between the Reduced Consensus ALADIN and Consensus ADMM in Appendix \ref{sec: Comparison}.
	\end{remark}
	
	\begin{algorithm*}
		\small
		\caption{Consensus BFGS   ALADIN}
		\textbf{Initialization:} choose $\rho>0$, initial guess $(\lambda_i,z,B_i\succ0)$ (or set $B_i= \rho I$). \\
		\textbf{Repeat:}
		\begin{enumerate}
			\item Each agent optimizes its own variable $x_i$ locally and transmit it to the master
			\begin{equation}\label{eq: NLP}
				{x_i}^+=\mathop{\arg\min}_{x_i} f_i(x_i)+\lambda_i^\top  x_i+\frac{\rho}{2}\|x_i-z\|^2
			\end{equation}
			with $\lambda_{i}=B_i(x_i^--z)-g_i$.
			
			\item  Decode the gradients and Hessian of each sub-problem at the master side.\\
			a) The master encode the  (sub)gradient and BFGS Hessian  from each ${x_i}^+$.
			\begin{equation}\label{eq: subgradient}
				\left\{
				\begin{split}
					g_i(x_i^+)=&\;\rho(z-x_i^+)-\lambda_i\quad& \text{ (local (sub)gradient evaluation) },\\
					s_i(x_i^+,x_i^-)=&\;x_i^+-x_i^-\quad &\text{ (difference of private variables) },\\
					y_i(x_i^+,x_i^-)=&\; g_i(x_i^+)-g_i^-\quad& \text{(difference of local (sub)gradient)) }.
				\end{split}
				\right.
			\end{equation}
			
			b) \text{Modify the local gradient based on the following condition}
			
			\begin{equation*}
				\left\{ 
				\begin{split}
					&y_i=y_i + \theta (B_is_i-y_i)\; \text{where}\; \theta= \frac{0.2(s_i)^\top B_i s_i-(s_i)^\top y_i}{(s_i)^\top B_i s_i-(s_i)^\top y_i}, & \text{if}\;(y_i)^\top s_i\leq\frac{1}{5} (s_i)^\top B_i s_i\\
					&y_i=y_i,& \text{otherwise}.
				\end{split}
				\right.
			\end{equation*}

			c) \text{BFGS Hessian approximation evaluation:}
			\begin{equation*}
				B_i^{+}=\;B_i-\frac{B_i s_i s_i^\top B_i}{s_i^\top B_i s_i}+\frac{y_iy_i^\top}{s_i^\top y_i}.
			\end{equation*}
			\item The master solve the following coupled QP with updated $x_i^+$, $g_i$ and $B_i$.
			\begin{equation*}
				\begin{split}
					z^+_{\text{BFGS}}=&\left(\sum_{i=1}^N  B_i^{+} \right)^{-1}\left( \left( \sum_{i=1}^N B_i^{+}x_i^+\right) -  \left( \sum_{i=1}^N g_i(x_i^+)  \right) \right).
				\end{split}
			\end{equation*}
			Broadcast the global model $z$ to the agents. 
		\end{enumerate}
		\label{alg:BFGS ALADIN2}
	\end{algorithm*}
	%		Notice that, the aggregation step of Algorithm \ref{alg:BFGS ALADIN2} depends on the inverse of the aggregated Hessian. It's still tricky at higher subproblem dimensions. 
	%	In the absence of a certain rate of convergence, by setting $B_i=\rho I$ and aggregate with Equation \eqref{eq: reduced QP2}, Algorithm \ref{alg:BFGS ALADIN2} degrade into another structure so called \emph{Reduced Consensus ALADIN}. Clearly by lossing of the second order information of each $f_i$, the corresponding inverse of aggregrated Hessian can be simplified.
	%\begin{algorithm}[H]
	%	%	\small
	%	\caption{Reduced Consensus ALADIN}
	%	\textbf{Initialization:} choose $\rho>0$, initial guess $(\lambda_i,z)$. \\
	%	\textbf{Repeat:}
	%	\begin{enumerate}
		%		\item Each agent optimizer their variable $x_i$ locally and transmit it to the master
		%		\begin{equation}
			%			{x_i}^+=\mathop{\arg\min}_{x_i} f_i(x_i)+\lambda_i^\top  x_i+\frac{\rho}{2}\|x_i-z\|^2
			%		\end{equation}
		%		with $\lambda_{i}=\rho(x_i^--z)-g_i$.
		%		
		%		\item The master and agents can evaluate the new subgradient from each ${x_i}^+$.
		%		\begin{equation}
			%			g_i=\rho(z-x_i^+)-\lambda_i.
			%		\end{equation}
		%		\item The master solve the following coupled QP with updated $x_i^+$ and $g_i$
		%		\begin{equation}
			%			\begin{split}
				%				z^+=\frac{1}{N} \sum_{i=1}^{N} \left(x_i^+- \frac{g_i}{\rho}\right),
				%			\end{split}
			%		\end{equation}
		%		then broadcast the global model $z$. 
		%	\end{enumerate}
	%	\label{alg:Derivative-Free consensus ALADIN2}
	%\end{algorithm}

	In order to expand the reach of our proposed C-ALADIN, by meeting the DC problems in FL, we next introduce a novel way of applying Reduced Consensus ALADIN in this area.
	
	\section{FedALADIN}\label{sec: FedALADIN}

	To meet the DC problems in FL, directly adopting Consensus BFGS ALADIN or Reduced Consensus ALADIN is a straightforward approach to solving such problems. However, we will meet the following challenges: first, since the high dimension of private variables $x_i$s, we cannot benefit from the second order information; second, existing algorithms, when solving the DC problems in FL, does not carefully determine a training \textit{epoch}, which makes the solutions inexact. This challenge renders the failure to directly adopting our techniques that provide exact solutions. By meeting the two challenges, in this section, we carefully design a variant of Reduced Consensus ALADIN, named \texttt{FedALADIN}, that works well for the problems in FL.
	
	We sketch the key design of \texttt{FedALADIN} as follows: By following the conventions in Algorithm \ref{alg:FedADMM} and observing the structure of Equation \eqref{eq: reduced QP2}, each client transmits Equation \eqref{eq: encode} instead of $x_i$.
	\begin{equation}\label{eq: encode}
		w_i= \left(x_i- \frac{g_i}{\rho}\right).
	\end{equation}
	%	If the decoupled NLPs \eqref{eq: NLP} are solved with high accuracy, the corresponding gradient or subgradient can be computed as \eqref{eq: subgradient}. However is not, it can not be decoded by the couped QP step.
	%	Moreover, from the consideration of security issue we could even modify the information that each clients uploaded. Suppose each clients do one more step to encode the transmitted information as	
	%	Then the encoded information has no physical meaning. And even more subtly, from the dual properity in Equation~\eqref{eq: KKT}, the effect of pairing information can be offset only if and only if the attacker gets all the uploaded information as the server. With the new structure, the server doesn't have to maintain the gradients $g_i$s and dual variables $\lambda_i$s of each client. 
	%		It should be pointed out that Equatio ~\eqref{eq: encode} is a consideration of security issue. Another version of \texttt{FedALADIN} could somehow be implemented by maintaining the gradients $g_i$s and dual variables $\lambda_i$s on both sides of clients and servers.
	Next, we detail \texttt{FedALADIN} in Algorithm \ref{alg: FedALADIN}. Different from Algorithm~\ref{alg:FedADMM}, Algorithm~\ref{alg: FedALADIN} obtains the local optimizer by using the decoded $\lambda_i$ from the previous global model $z$. Then we evaluate the (sub)gradients with Equation~\eqref{eq: subgradient}.  Later we encode the transmitted date $w_i$s with Equation~\eqref{eq: encode}. After receiving $w_i$s from each client, the server aggregates them as \texttt{FedADMM}.
	%		\begin{algorithm}[H]
		%		%	\small
		%		\caption{\texttt{FedALADIN}: Derivative-Free Consensus ALADIN for FL}
		%		%\textbf{Initialization:} Initial guess $(p^0,\lambda^0,\mu^0)$, choose $\Sigma_i,\rho^0,\mu^0,\epsilon$. \\
		%		\textbf{Repeat:}
		%		\begin{enumerate}
			%			\item Each client learn their parameter $x_i$ locally and transmit it to the server
			%			\begin{equation}
				%				{x_i}^+=\mathop{\arg\min}_{x_i} f_i(x_i)+\lambda_i^\top  x_i+\frac{\rho}{2}\|x_i-z\|^2
				%			\end{equation}
			%			with $\lambda_{i}=\rho(x_i^--z)-g_i$.
			%			
			%			\item The clients can evaluate the new subgradient from each ${x_i}^+$.
			%			\begin{equation}\label{eq: subgradient}
				%				g_i=\rho(z-x_i^+)-\lambda_i.
				%			\end{equation}
			%			Then do local model regularizer and transmit to the server:
			%			\begin{equation}\label{eq: encode}
				%				w_i= \left(x_i^+- \frac{g_i}{\rho}\right). 
				%			\end{equation}
			%			
			%			\item Broadcast:
			%			\begin{equation}
				%				z = \frac{1}{N} \sum_{i=1}^{N} w_i.
				%			\end{equation} 
			%		\end{enumerate}
		%		\label{alg: FedALADIN}
		%	\end{algorithm}
	
	\begin{algorithm}[H]
		\small
		\caption{\texttt{FedALADIN}}
		\textbf{Initialization:} Initial guess  of global model $z=0$, local model $x_i^-=0$, gradient $g_i=0$ and  dual variables $\lambda_i=0$. Set the total number of rounds $T$ and  penalty parameter $\rho$.
		\vspace{0.1em}	
		
		\textbf{For $t=1\dots T$}
		\texttt{Clients:} // In parallel\\
		\textbf{ \hspace*{0.5em} For $i=1\dots N$}\\
		%	\begin{enumerate}
			\hspace*{1em}Download $z$ from the server\\
			\hspace*{0.8em}  Locally update $ w_i\leftarrow  \texttt{ClientUpdate} (z,i)$\\
			\hspace*{1.1em}Upload $w_i$ to the server\\
			%	\end{enumerate}
		\textbf{ \hspace*{0.3em} End}
		
		\texttt{Server:}
		$
		z = \frac{1}{N} \sum_{i=1}^{N} w_i.
		$\\
		\textbf{End}

		\vspace{0.1em}	
		$\texttt{ClientUpdate} (z,i)$:\\
		\textbf{Input:} Local epoch number $E_i$, client learning rate $\eta_i$.\\
		%	\begin{enumerate}
			\hspace*{0.5em}$\lambda_{i}=\rho(x_i^--z)-g_i$.
			%	\end{enumerate}
		
		\hspace*{0.5em}	\textbf{For $e=1\dots E_i$}
		
		%\begin{enumerate}
		\hspace*{1em}	${x_i}=x_i-\eta_i\left(  \nabla f_i(x_i)+\lambda_i+\rho\left( x_i-z \right) \right)$
		%	\end{enumerate}
	
	\hspace*{0.5em}	\textbf{End}
	
	%	\begin{enumerate}
		\hspace*{0.5em} $	g_i= \nabla f_i(x_i)$\\%\rho(z-x_i)-\lambda_i
		\hspace*{0.5em}	 $w_i=x_i- \frac{g_i}{\rho}$
		%	\end{enumerate}
	
	\textbf{return:} $w_i$
	
	\label{alg: FedALADIN}
\end{algorithm}

Note that the major difference between \texttt{FedALADIN} and \texttt{FedADMM} is the way of dual update. More importantly, the gradient evaluation \eqref{eq: subgradient} is a symmetric operation as the dual update of Equation~\eqref{eq: ADMM1}. %,  while the later steps are almost the same.
However, \texttt{FedALADIN} can benefit from the reduced QP operation compare with the former. In terms of structure, all the existing algorithms in FL are special cases of \texttt{FedALADIN}.
\begin{figure}[H]
	\centering
	\includegraphics[width=0.45\textwidth,height=0.17\textheight]{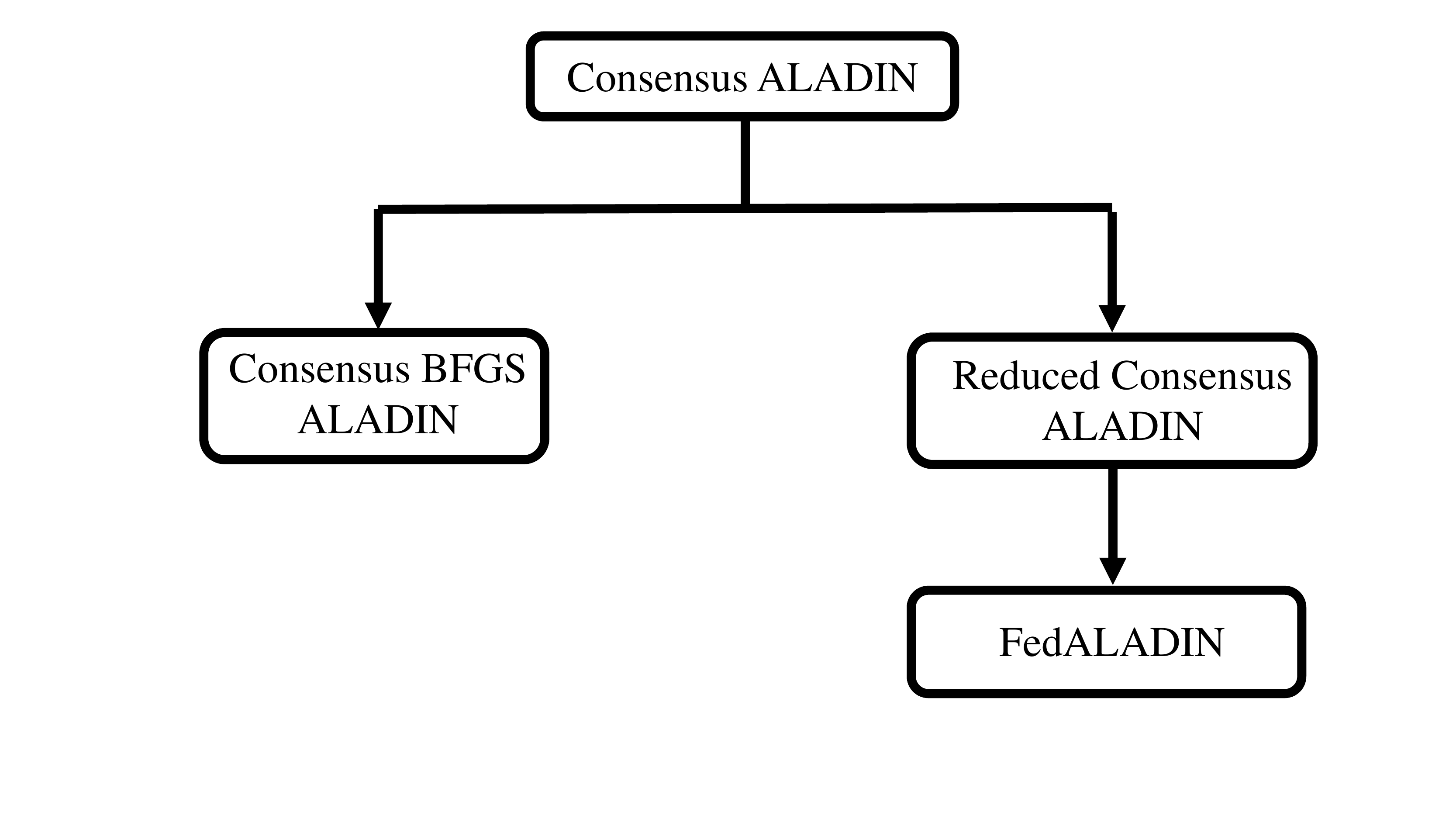}
	\caption{Consensus ALADIN family.}
	\label{fig: CONSENSUS-ALADIN}
\end{figure}

In the next section, we aim to establish the global and local convergence theory of C-ALADIN.

%\newpage

\section{Convergence Analysis}\label{sec: convergence}
In this section, we are interested in the convergence behavior of C-ALADIN, which consists of the following three parts:  
\begin{itemize}
	\item Global convergence of Reduced Consensus ALADIN for convex problems (Subsection \ref{sec: global convergence}).
	\item Global linear convergence rate of Reduced Consensus ALADIN for Lipschitz continuous or strongly convex problems %and some technical assumptions 
	(Subsection \ref{sec: convex rate}).
	\item Local convergence analysis of non-convex problems with C-ALADIN (Subsection \ref{sec: local convergence}).
\end{itemize}

The following equations will be used several times in the following proof. They are provided here for convenience.
% although have appeared
% in Equation~\eqref{eq: reduced QP}, \eqref{eq: subgradient} and \eqref{eq: RCA}.
%	\begin{equation}
	\begin{align}
		g_i&=\rho(z-x_i^+)-\lambda_i,\label{eq: gradient}\\ 
		\lambda_{i}^+&=\rho(x_i^+-z^+)-g_i,\label{eq: lam_plus}\\
		\sum_{i=1}^{N} \lambda_i &= 0,\label{eq: sum_lam} \\ 
		z^{+}&=\frac{1}{N} \sum_{i=1}^{N} \left(x_i^+- \frac{g_i}{\rho} \right),\label{eq: z_plus}  \\
		\sum_{i=1}^{N} x_i^+&=\frac{N}{2}(z^++z), \label{eq: sum_x}\\
		2(a-c)^\top (a-b) &=\|a-c\|^2-\|b-c\|^2+\|a-b\|^2 .\label{eq: quadratic equality}
	\end{align}
	Note that, by plugging Equation~\eqref{eq: gradient} and \eqref{eq: lam_plus} into \eqref{eq: z_plus}, \eqref{eq: sum_x} can be obtained. Equation~\eqref{eq: quadratic equality} comes from \cite{shi2014linear}. 
	
	The following lemma will also be useful in our proofs.
	
	\begin{lemma}
		With the procedure of Algorithm~\ref{alg:FedADMM}, the local primal update has a relationship with the local dual and global primal variables in the following way
		\begin{equation}\label{eq: primal}
			x_i^+ = \frac{\lambda_i^+-\lambda_i}{2\rho}+\frac{z^++z}{2}.
		\end{equation} 
	\end{lemma}
	\begin{proof}
		From Equation~\eqref{eq: lam_plus},
		\begin{equation*}
			\begin{split}
				\lambda_{i}^+&=\rho(x_i^+-z^+)-g_i\\
				&\overset{\eqref{eq: gradient}}{=}\rho(x_i^+-z^+)-\rho(z-x_i^+)+\lambda_i\\
				&=\rho(2x_i^+-z^+-z) +\lambda_i\\\
				\Longleftrightarrow\quad&  x_i^+ = \frac{\lambda_i^+-\lambda_i}{2\rho}+\frac{z^++z}{2}. 
			\end{split}
		\end{equation*}
	\end{proof}

	\subsection{Global Convergence of Convex Case}\label{sec: global convergence}
	
	The following global convergence proof relies on the \emph{Lyapunov stability theory} in the control community. A relationship between the former and the latter has been clarified in Appendix \ref{App: Lya}. %Moreover, it is the reason that \cite{Boyd2011,shi2014linear, ling2015dlm} also prove the global convergence of ADMM by applying different Lyapunov functions.
	
	We assume that the sub-functions $f_i$s are closed, proper, and strictly convex. %We stress that the third condition also implies the uniqueness of $z^{*}$.
	For establishing the global convergence theory of Reduced Consensus ALADIN,		
	we introduce the following \emph{Lyapunov function} \cite{nonlinear} with the global minimizer $z^*$.
	%	\[\mathscr{L}_B(z,\lambda)=\frac{1}{\rho} \sum_{i=1}^{N}\|\lambda_i-\lambda_i^* \|^2 + \rho\|z-z^*\|_{\left(\sum_{i=1}^N  B_i \right)}^2.\]
	%	When we choose $B_i=\rho I$, it reduced to
	\begin{equation}
		\mathscr{L}(z,\lambda)=	\frac{1}{\rho} \sum_{i=1}^{N}\|\lambda_i-\lambda_i^* \|^2 + \rho N  \|z-z^*\|^2.
	\end{equation}
	Note that the choice of Lyapunov function is not unique. Next, we will prove the global convergence of Reduced Consensus ALADIN by showing that the Lyapunov function is monotonically decreasing.

	\begin{theorem}\label{The: 1}
		Suppose $f_i$s are strictly convex and problem~\eqref{eq: FL} has an existing solution $z^*$, then
		\begin{equation}
			\mathscr{L}(z,\lambda)-\mathscr{L}(z^+,\lambda^+)\geq \alpha\left(\|x_i^+-z^*\|\right)\geq0
		\end{equation}
		will hold by applying Reduced Consensus ALADIN.
		Here, $\alpha$ is a class $\mathcal K$ function \cite{nonlinear}.		
	\end{theorem} 
	\begin{proof}
		See Appendix \ref{APP: C}.
	\end{proof}
	
	According to Theorem~\ref{The: 1}, the global convergence of Reduced Consensus ALADIN can be established. In order to prove the convergence of sequence $(z, \lambda)$ to the global optimal solution pair $(z^*, \lambda^*)$, we need to show the uniqueness of it.
	
	\begin{theorem}\label{Them: uniqueness}
		We assume that Theorem~\ref{The: 1} holds, which yields
		\begin{equation}\left\{
			\begin{split}
				\lim_{k\rightarrow\infty} z^k &= z^*\\
				\lim_{k\rightarrow\infty} \lambda^k &= \lambda^*,\\
			\end{split}\right.
		\end{equation}
		where $k$ denotes the index of iterations.
	\end{theorem}
	\begin{proof}
		See Appendix \ref{APP: D}.
	\end{proof}
	
	The above two theorems show the convergence of $\left( z, \lambda\right)$. 
	Later, we will show  $x_i$  is also convergent in the following theorem.%, has been made.    
	% for $x_i$ is also convergent to discussed.
	\begin{theorem}\label{them: local model}
		If Theorem~\ref{The: 1} holds, then we have $x_i\rightarrow z^*$. 
	\end{theorem}
	\begin{proof}
		See Appendix \ref{APP: E}.
	\end{proof}

	% \newpage 
	From the above three theorems, a convergence of the sub-gradients of the agents can be also easily established.
	\begin{theorem}\label{Them: gradient converge}
		We assume that Theorem~\ref{The: 1}, \ref{Them: uniqueness}, and \ref{them: local model} hold jointly. Then, $g_i$ converges to $-\lambda_i^*$ globally. 
	\end{theorem}
	\begin{proof}
		See Appendix \ref{APP: F}.
	\end{proof}

	Note that the global convergence proof only requires the strict convexity of objectives without smoothness and strongly convexity assumptions. 
	
	C-ALADIN, same as T-ALADIN, in case $ f_i(x_i)$s are only convex rather than strictly convex, guarantees that the solutions can converge to an optimal set instead of a single optimal solution, which will be theoretically analyzed in Theorem~\ref{theorem 2 }.

	% also has no guarantee for converging to a single optimal solution but

	% an optimal set if  $\sum_{i=1}^{N} f_i(x_i)$ is convex without strictly convex assumption.
	
	\begin{theorem}\label{theorem 2 }
		Suppose that $\mathbb Z^*$ denotes the set of optimal primal solution and $\mathbf\Lambda_i^*$ represents the optimal dual set,  with $\lambda_i^*\in \mathbf\Lambda_i^*$, we have
		$z$ converge to $\mathbb Z^*$ globally,
		\[\lim_{z^*\in \mathbb Z^*}\left\|z -z^* \right\|=0.\] 
	\end{theorem}
	\begin{proof}
		Here the second auxiliary function in Equation~\eqref{eq: aux} will be reduced to
		\begin{equation}
			G(\xi)=\sum_{i=1}^{N}f_i(\xi_i)+\sum_{i=1}^{N}(\xi_i-z^*)^\top\lambda_i^*
		\end{equation}
		with $\lambda_i^*\in \mathbf\Lambda_i^*$ that $\mathbf\Lambda_i^*$ is the optimal set of each dual variables which means that they are not unique. In this way, Theorem~\ref{Them: uniqueness} is not used in this case.
		
		However, if we apply the proof of Theorem~\ref{The: 1} again here, the monotone decreasing of the Lyapunov function still holds. In this case, we have a similar result as \cite{Houska2021}.
	\end{proof}
	
	Note that, no global convergence rate of convex optimization has been discussed in the T-ALADIN research \cite{Houska2021}. 
	As an additional contribution compared with the former, in the next subsection we will find the convergence rate 
	of convex cases by using Reduced Consensus ALADIN with some extra technical assumptions.

	\subsection{Global Linear Convergence Rate Analysis}\label{sec: convex rate}
	
	Next we will prove the Q-linear convergence rate \cite{Nocedal2006} by adding additional $m_f$ strongly convex or $\omega_f$ smooth assumptions.
	\begin{theorem}\label{theorem 3 }
		Suppose that $\sum_{i=1}^{N} f_i(x_i)$ is $m_f$ strongly convex, and there exists a $\delta>0$ such that
		\begin{equation}
			\delta \mathscr{L}(z^+,\lambda^+)\leq 4 m_f \sum_{i=1}^{N}\left\| x_i^+ -z^*\right\|^2,
		\end{equation}
		then C-ALADIN is Q-linearly converging to a unique optimal solution with rate $\left( \frac{1}{\sqrt{1+\delta}} \right)$.
	\end{theorem} 
	\begin{proof}
		See Appendix \ref{APP: G}.
	\end{proof}
	
	In the Reduced Consensus ALADIN, we find the condition of $m_f$ strongly convex and $\omega_f$ smooth are symmetric. Q-linearly convergent results can also be established later.
	\begin{corollary}\label{corollary}
		Suppose that $\sum_{i=1}^{N} f_i(x_i)$ is $\omega_f$ smooth and convex, we have a similar result as Theorem \ref{theorem 3 },  there exists a $\delta>0$ such that
		\begin{equation}\label{eq: strongly convex}
			\delta \mathscr{L}(z^+,\lambda^+)\leq \frac{4 }{\omega_f}\sum_{i=1}^{N}\left\| g_i -g_i^*\right\|^2,
		\end{equation}
		then Reduced Consensus ALADIN can also Q-linearly converge to a unique optimal solution with a rate $\left( \frac{1}{\sqrt{1+\delta}} \right)$. 	
	\end{corollary}
	\begin{proof}
		From the definition of $\omega_f$ Lipschitz continuous \cite{Nocedal2006}
		\begin{equation}
			\|g_i -g_i^*\|\leq\omega_f\|x_i^+-z^*\|,
		\end{equation}
		the following inequality can be obtained
		\begin{equation}
			\begin{split}
				\frac{1}{\omega_f} \sum_{i=1}^{N}\|g_i -g_i^* \|^2\leq  \sum_{i=1}^{N} \left(x_i^+-z^* \right)^\top \left( g_i -g_i^* \right).
			\end{split}
		\end{equation}
		Note that the right-hand side is the same equation as that of Equation \eqref{eq: linear rate}.
		Such that, the later proof is similar to Theorem \ref{theorem 3 } and is not shown here.
	\end{proof}
	
	Therefore, Reduced Consensus ALADIN needs either $m_f$ strongly convex or $\omega_f$ Lipschitz continuous to establish the global Q-linear convergence theory. 
	
	In the above two sections, we show the global convergence of Reduced Consensus ALADIN has a similar property as ADMM.
	But different from the latter, in the next section we will show local convergence analysis of non-convex cases of C-ALADIN.
	
	\subsection{Local Convergence Analysis of Non-convex Case}\label{sec: local convergence}
	The following convergence analysis depends on the assumption that $f_i$s are twice continuously differentiable in a neighborhood of a local minimizer $z^*$.
	To benefit from the theory of SQP \cite[Chapter 18]{Nocedal2006}, in the corresponding convergence analysis of C-ALADIN, we introduce $\gamma$ as the upper bound of Hessian approximation difference $\left\|B_i-\nabla^2f_i(x_i)\right\|\leq \gamma$.

	% The corresponding convergence analysis of C-ALADIN fully inherits from SQP theory \cite[Chapter 18]{Nocedal2006} by introducing $\gamma$ as the upper bound of Hessian approximation difference $\left\|B_i-\nabla^2f_i(x_i)\right\|\leq \gamma$.% (Appendix \ref{appendix}).
	
	\begin{theorem}\label{them: local convergence}
		If $f_i$s are non-convex, C-ALADIN can still converge with sufficient large $\rho$ with different convergence rate in different situations of $\gamma$.
	\end{theorem}
	\begin{proof}
		See Appendix \ref{APP: H}.
	\end{proof}
	
	Next, we will show the numerical performance of C-ALADIN.		
	
	%	Note that the local convergence analysis in Subsection~\eqref{sec: local convergence} can be also applied in the global convergence rate analysis in Subsection~\eqref{sec: convex rate}, since convex problems can be treated as special cases of nonconvex problems. However, what we have proposed in Subsection~\eqref{sec: local convergence} shows that a similar way of analysing convergence rate can be also applied in the convex ALADIN framework without depending on the second order information of each subproblems. This shows a similarity between ALADIN and ADMM.
	%	From theory part of view, we show that the convergence result of ALADIN is at least as strong as ADMM in the convex world, but for nonconex cases we can show much more.

	\section{Numerical Experiments}\label{sec: Numerical}
	In this section, we illustrate the numerical performance of the proposed algorithms on both distributed optimization (Subsection \ref{sec: DC}) and learning (Subsection \ref{sec: FL case}). 
	
	% As mentioned in Section \eqref{sec:introduction} and \eqref{sec: FedALADIN}, C-ALADIN was originally designed for optimization problems and has not been specifically tailored for use in federated learning. However, it is still worth examining its performance in this context, even in the absence of theoretical guarantees and revisions.

	\subsection{ Case Studies on Distributed Consensus Optimization }\label{sec: DC}
	
	In this subsection, all the implementation of algorithms relies on \texttt{Casadi-v3.5.5} with \texttt{IPOPT} \cite{Andersson2019}. 
	
	The first case is a convex consensus least square problem
	\begin{equation}\label{eq: convex least square}
		\begin{split}
			\min_{x_i ,z}\;\;&\frac{1}{2} \sum_{i=1}^{N}\|x_i- \zeta_i\|_2^2 \\ \quad\mathrm{s.t.}\;\;& \;x_i = z\  |\lambda_i. \;
		\end{split}
	\end{equation}
	Here, $x_i\in \mathbb R^{100}$ and $N=200$. The measured data $\zeta_i$s are drawn from the Gaussian distribution $\mathcal N(0, 25) $. In this setting, Problem \eqref{eq: convex least square} has $20100$ primal variables and  $20000$ dual variables, which is a large-scale optimization problem. 
	\begin{figure}[htbp]
		\centering
		\includegraphics[width=0.4\textwidth,height=0.23\textheight]{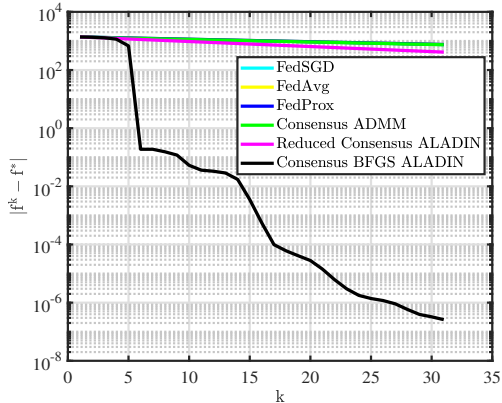}
		\caption{Numerical comparison on \textbf{convex} case.}
		\label{fig:Convex}
	\end{figure}
	In our implementation, the learning rate of \texttt{FedSGD} $\eta_i$s are set as $0.01$ while other compared algorithms can update the local primal variables with \texttt{Casadi} exactly. Moreover, the hyper-parameter $\rho$ is set as $10^2$ for other algorithms. Note that, all the initial values of primal and dual variables are set as zeros vectors.
	In the optimization framework, we assume that all the other algorithms can be solved exactly with local optimizer.

	As the same setting, a distributed non-convex optimization problem can be easily implemented as \eqref{eq: nonconvex problem}. Note that, excluding the second term of the objective function, the non-convex optimization problem is directly reduced to the convex one \eqref{eq: convex least square}.
	\begin{equation}\label{eq: nonconvex problem}
		\begin{split}
			\min_{x_i ,z}\;\;& \sum_{i=1}^{N}\frac{1}{2}\left(\|x_i^a- \zeta_i^a\|_2^2+\|x_i^b- \zeta_i^b\|_2^2\right)\\
			&\quad\;\;+\text{log}\left(\frac{1}{2}\|(x_i^a-x_i^b)^2-\zeta_i^c\|_2^2\right) \\ \quad\mathrm{s.t.}\;\;&\; x_i = z\  |\lambda_i \;
		\end{split}
	\end{equation}
	with $x_i=[(x_i^a)^\top,(x_i^b)^\top]^\top$.

	\begin{figure}[htbp]
		\centering
		\includegraphics[width=0.4\textwidth,height=0.23\textheight]{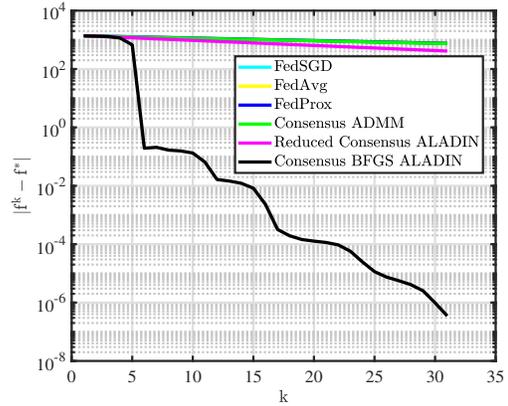}
		\caption{Numerical comparison on on \textbf{nonconvex} case.}
		\label{fig:nonconvex}
	\end{figure}
	
	Figure \eqref{fig:Convex} and \eqref{fig:nonconvex} illustrate the numerical convergence comparison among several numerical algorithms  for Problem \eqref{eq: convex least square} and \eqref{eq: nonconvex problem} with iteration $k$. Consensus BFGS ALADIN performs far superior to other algorithms according to the given two numerical optimization cases (convex and non-convex). The performance of Reduced Consensus ALADIN is also better than other existing distributed algorithms. We didn't compare with all the existing algorithms such as Douglas-Rachford splitting (DRS) since they have similar performance as ADMM in practice. %In some cases, they can be regarded as special cases of ADMM. 
	In both cases, Consensus BFGS ALADIN
	can get a high accuracy ($10^{-4}$) within $20$ iterations. This shows the importance of second order information in distributed consensus optimization.

	From the above results, we can easily observe the convergence trend of all the evaluated algorithms. We want to stress that other algorithms do not have a generalized convergence guarantee for non-convex problems.

	%\begin{equation*}
	%	\begin{split} 
		%		\min_{x_1, x_2}\;\;&\frac{1}{2} \|Ax_1 - b \|^2 +\kappa \|x_2\|_1 \\ \quad\mathrm{s.t.}\;\;& x_1 = x_2 \  |\lambda. \;
		%	\end{split}
	%\end{equation*}
	%which is equilivent to 
	%\begin{equation*}
	%	\begin{split} 
		%		\min_{x_1, x_2, z}\;\;&\frac{1}{2} \|Ax_1 - b \|^2 +\kappa \|x_2\|_1 \\ \quad\mathrm{s.t.}\;\;& x_1 = z\  |\lambda_1 \;\\
		%		& x_2 = z\  |\lambda_2. \;
		%	\end{split}
	%\end{equation*}
	%Here $A\in \mathbb{R}^{10\times 100}$, $b\in \mathbb R^{10}$.

	\subsection{Case Studies on Federated Learning}\label{sec: FL case}
	In this subsection, we illustrate the numerical performance of the proposed algorithm in FL. 
	Consensus BFGS ALADIN is not compared with other algorithms in this subsection since second-order information is rarely used in the community of FL.
	For fairness, our numerical comparison relays on the implementation of recent work \cite{zhou2022federated} with released code
	\url{https://github.com/ShenglongZhou/FedADMM}. Here, we mainly focus on the convergence rate and stability of the algorithm.

	Different from the standard Consensus ADMM, the hyperparameter
	$\rho_i$s in \cite{zhou2022federated} are set to a different value, which means that the learning rate of each sub-problem will be affected differently. Same as \cite{zhou2022federated}, here we first compared performance algorithms on a convex linear regression problem with local objectives \eqref{eq: linear regression} and non-i.i.d. data.
	\begin{equation}\label{eq: linear regression}
		f_i(x) = \sum_{t\in \mathcal D_i} \frac{1}{2d_i} ((a_i^t)^\top x-b_i^t)^2.
	\end{equation}
	Here $a_i^t\in \mathbb R^{100}$ and $b_i^t\in \mathbb R$ are the $t$-th sample data of client $i$. We set $x\in \mathbb R^{100}$ and $N=100$.
	Another example is a non-convex logistic regression  with sub-objectives \eqref{eq: logistic regression}.
	\begin{equation}\label{eq: logistic regression}
		f_i(x)=\frac{1}{|\mathcal D_i|}\sum_{t\in \mathcal D_i} \left( \text{ln}\left( 1+e^{(a_i^t)^\top x}\right)-b_i(a_i^t)^\top x\right) + \frac{\lambda}{2}\|x\|^2.
	\end{equation}
	Here $a\in \mathbb R^{1024}$, $b\in \{0,1\}$,
	$\lambda=0.001$, $x\in \mathbb R^{1024}$ and also with $100$ clients. Importantly, each client has a participation rate and is set as $0.1$. %Although C-ALADIN didn't have these specific considerations in mind.
	The rest of the technical details can be referred to \cite[Section 5]{zhou2022federated}. For fairness, we have no other updates.
	\begin{figure}[htbp]
		\centering
		\begin{minipage}{0.49\linewidth}
			\centering
			\includegraphics[width=1\linewidth]{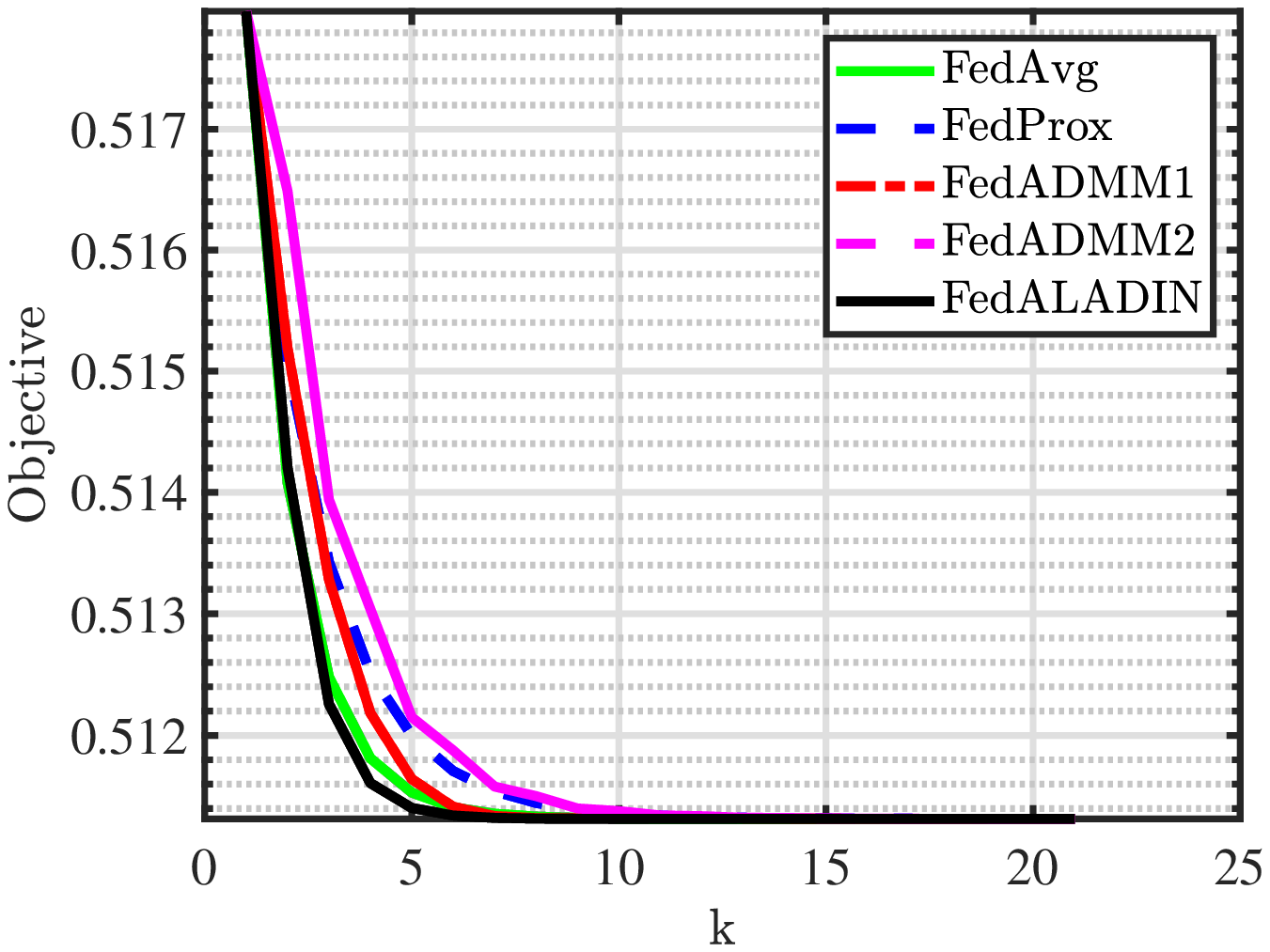}
			\caption{Linear regression with different $\rho$ (\textbf{convex} problem).}
			\label{fig: diff_rho_convex}%文中引用该图片代号
		\end{minipage}
		\begin{minipage}{0.49\linewidth}
			\centering
			\includegraphics[width=1\linewidth]{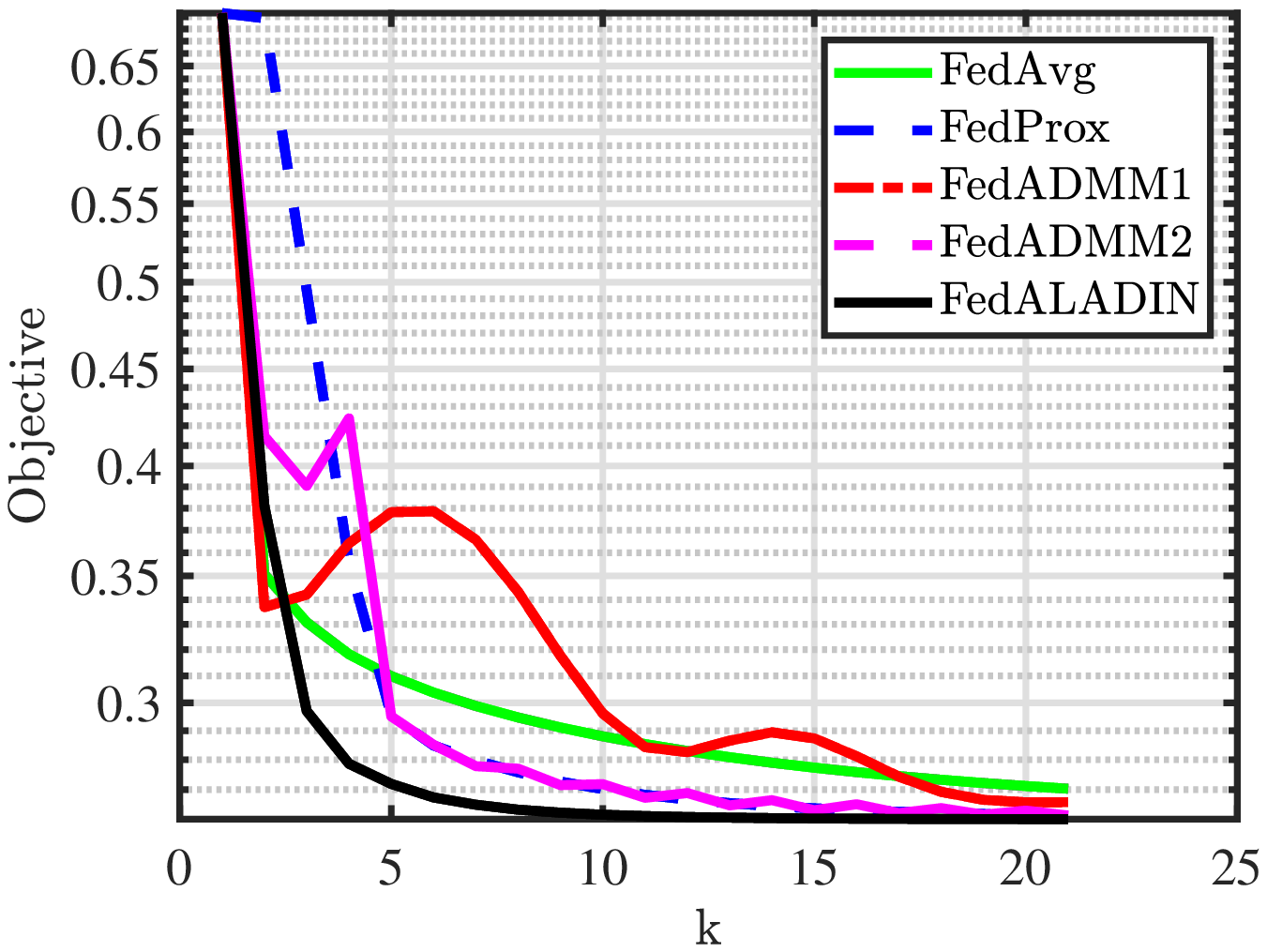}
			\caption{Logistic regression with different $\rho$ (\textbf{nonconvex} problem).}
			\label{fig: diff_rho_nonconvex}%文中引用该图片代号
		\end{minipage}
	\end{figure}
	
	Figure \eqref{fig: diff_rho_convex} and \eqref{fig: diff_rho_nonconvex} show the numerical comparison on the linear regression and logistic regression problems respectively.   
	Here, \texttt{FedADMM1} represents the method proposed by \cite{zhou2022federated} that update the dual before global aggregation. \texttt{FedADMM2}, follows the order as \cite{boyd2011distributed}. The convergence performance of the two schemes is very different in logistic regression problem. 
	It has to be pointed out that, in the optimization setting,  there is only a little difference in performance between the two ADMM variations mentioned above, but in this experiment, we can see that they behave completely differently. This once again illustrates the fundamental difference between locally exact or inexactly search of the local primal variables.   While this paper does not address this issue theoretically in C-ALADIN family, it still points out that it is an open question worthy of attention.

	\begin{figure}[htbp]
		\centering
		\begin{minipage}{0.49\linewidth}
			\centering
			\includegraphics[width=1\linewidth]{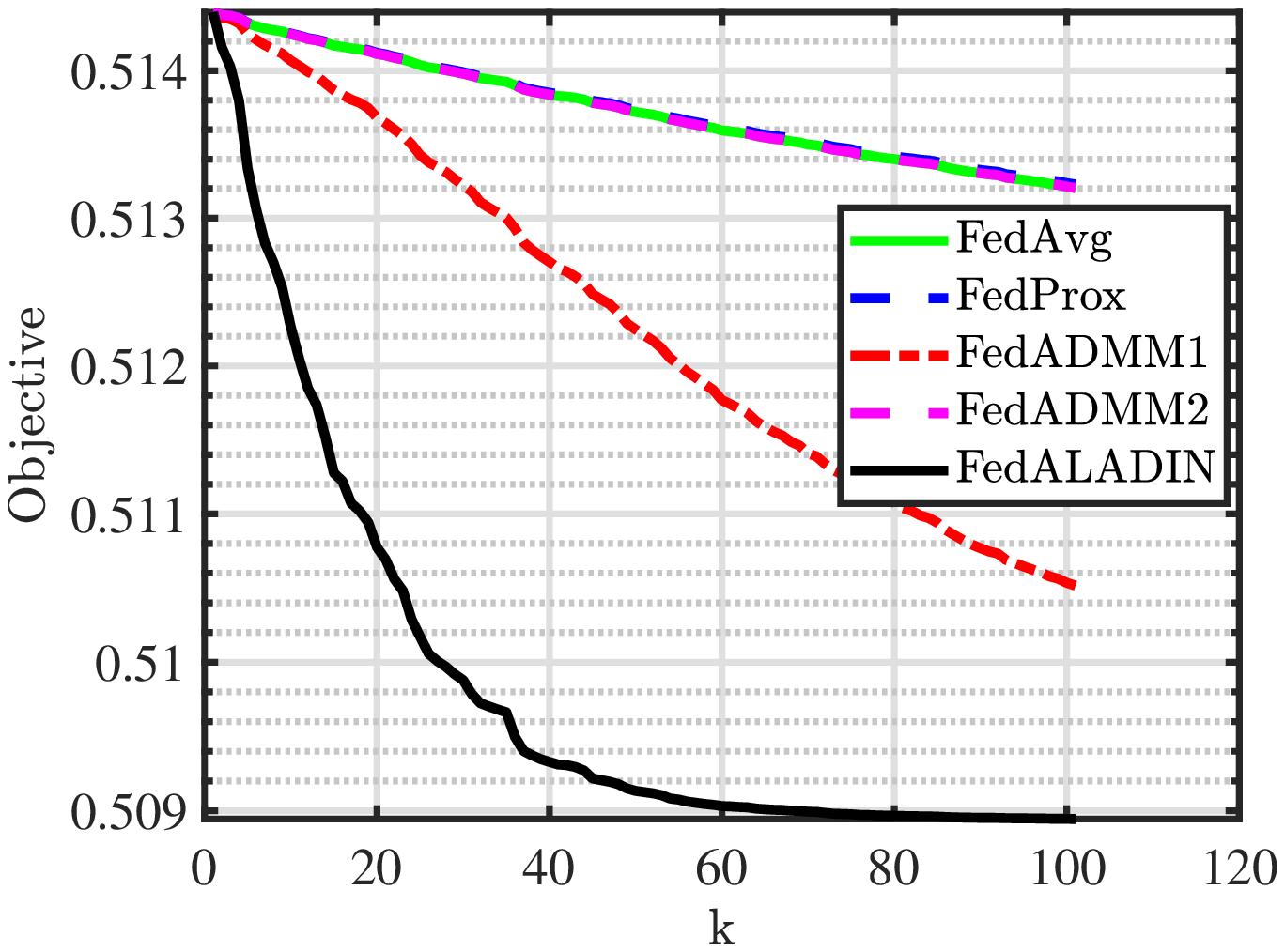}
			\caption{Linear regression with the same $\rho$  (\textbf{convex} problem).} 
			\label{fig: same_rho_linear}%文中引用该图片代号
		\end{minipage}
		\begin{minipage}{0.49\linewidth}
			\centering
			\includegraphics[width=1\linewidth]{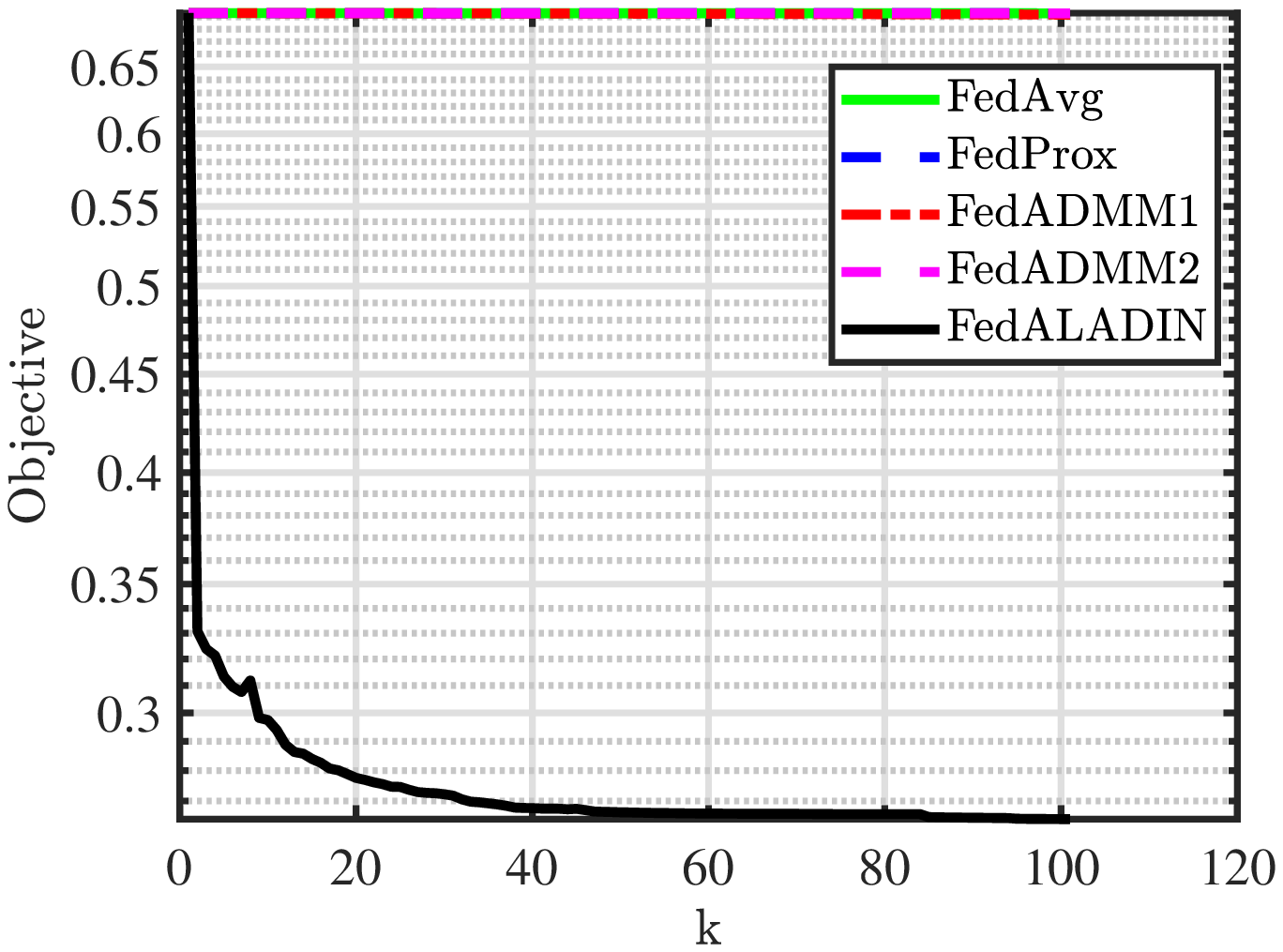}
			\caption{Logistic regression with with same $\rho$ (\textbf{nonconvex} problem).}
			\label{fig: same_rho_nonconvex}%文中引用该图片代号
		\end{minipage}
	\end{figure}
	For further discussion, as shown in Figure \eqref{fig: same_rho_linear} and \eqref{fig: same_rho_nonconvex}, we find that when the hyperparameter $\rho$ is set to the same value ($\rho=0.5$ for linear regression and $0.05$ for logistic regression). With a learning rate $\eta=0.01$ for all the algorithms, the convergence performance of \texttt{FedALADIN} is far better than the existing algorithms at least in these two examples. In addition, we found that \texttt{FedADMM2} is not stably converged in the first several iterations.
	\section{Related Work}
	\label{related-work}
	Existing distributed convex optimization algorithms can be roughly divided into two types: primal decomposition (PD)
	\cite{geoffrion1970primal,demiguel2008decomposition,engelmann2022scalable} and DD (also called Lagrangian decomposition).
	We refer  \cite{decomposition,terelius2010distributed,pflaum2014comparison,palomar2006tutorial}
	as references for more details.
	
	PD aims to partition the problem in a lower-upper level fashion, where the upper-level problem considers the lower-higher level problems by their optimal value functions which means control the private variables directly. 
	Different from the former, the higher level problem influence the lower level ones by using dual variables (shadow price) in the DD structure. %Usually the decision variables are always feasible at each iteration with PD. On the contrary,  feasibility satisfies only after convergence with DD. 
	To the best of our knowledge, only few literature  studied the theoretical comparison between PD and DD.
	%convergence rate and stability of the two, and in different scenarios, the performance of the two has winner and loser.
	However, numerically, \cite{terelius2010distributed,pflaum2014comparison} showed that DD performances a better convergence rate compare with PD, however the degree of stability of convergence is the opposite in some applications. A discussion can be found in \cite[Section I]{ling2015dlm}.  The efforts in PD and DD can be further categorized into the following two fashions, namely exact search  and inexact search. Exact search comes from the optimization community while inexact search is drawn from the FL community.
	
	In the PD family, representative algorithms of exact search  are DGD \cite{yuan2016convergence} and EXTRA \cite{shi2015extra}.
	On the inexact search side, \texttt{FedSGD} \cite{ioffe2015batch}, \texttt{FedAvg} \cite{mcmahan2017communication}, \texttt{FedProx} \cite{li2020federated} and \texttt{FedDANE} \cite{li2019feddane} were proposed.

	In the DD family, current techniques, consisting of consensus ADMM \cite{boyd2011distributed} and DD, only focus on distributed convex problems with exact search. In contrast, our proposed C-ALADIN has guarantee for non-convex problems.
	On the inexact search side, the state-of-the-art algorithm is \texttt{FedADMM} \cite{zhou2021communication,zhou2022federated,wang2022confederated}. Our proposed \texttt{FedALADIN} extends this line of work by showing a more stable convergence performance than \texttt{FedADMM}.

	% shows a better convergence performance compare    Moreover,  illustrate a bi-level ADMM based federated learning framework for multiple servers that which is an extension of all the previous frameworks  ..

	% We next categorize existing work from the aspects of exact search and inexact search. 
	
	% Exact search follows the convention of optimization community, which consists of PD (\cite{yuan2016convergence} and EXTRA \cite{shi2015extra}) and DD (Consensus ADMM  and C-ALADIN).
	% \begin{figure}[H]
		% 	\centering
		% 	\includegraphics[width=0.43\textwidth,height=0.17\textheight]{Img/Primal-Dual}
		% 	\caption{Primal decomposition (PD) and dual decomposition (DD) algorithms.}
		% 	\label{fig:Primal-Dual}
		% \end{figure}
	
	% Note that, decentralized optimization algorithms such as DGD \cite{yuan2016convergence} and EXTRA \cite{shi2015extra}  can be treated as consensus special cases of PD.  
	% As a well-known distributed algorithm, ADMM combined DD with augmented Lagrange method (ALM). In this way the dual update has an explicit form which induces the convergence of the algorithm for convex problems more stable in practice \cite{boyd2011distributed}. Although  \cite{hong2016convergence,wang2019global} explores some special cases where ADMM can be used and converges for non-convex problems, it is still a huge challenge for ADMM in general and so do the class of PD. 

	\begin{table}[H]
		%\caption{values of numerical experiments}\\
		\renewcommand{\arraystretch}{1.5}
		\caption{Existing and our proposed algorithms for DC.}
		\label{table:comparison2}
		\begin{center}
			%\begin{tabular}{|c||p{0.9cm}|p{1cm}|p{1.2cm}|p{1.3cm}|}
			\begin{tabular}{|c|c|c|}
				\hline
				Methods	&\makecell{Primal Decomposition}& \makecell{Dual Decomposition}\\
				\hline
				Exact    &\makecell{DGD \cite{yuan2016convergence},}  &\makecell{Consensus ADMM \cite{Boyd2011},   } \\
				Search   &\makecell{ EXTRA \cite{shi2015extra}, NIDS \cite{li2019decentralized}}  &\makecell{   DD \cite{everett1963generalized}, {\textbf{C-ALADIN}}} \\
				\hline
				Inexact    &\makecell{\texttt{FedSGD}\cite{ioffe2015batch}, \texttt{FedDANE} \cite{li2019feddane}, }  &\makecell{\texttt{FedADMM} \cite{zhou2022federated,wang2022confederated}, } \\
				Search   &\makecell{\texttt{FedProx} \cite{li2020federated}, \texttt{FedAVG}\cite{mcmahan2017communication}}  &\makecell{{\textbf{\texttt{FedALADIN}}} } \\
				\hline
			\end{tabular}
		\end{center}
	\end{table}

	\section{Conclusion $\&$ Outlook}\label{sec: conclusion}
	This paper proposed a novel distributed consensus algorithm family, named C-ALADIN, that is efficient in solving non-convex problems. 
	In the framework of C-ALADIN, we proposed efficient structure for communication and computation efficiency. Based on the framework, depending on whether second order information is used, two variants of C-ALADIN are proposed, named Consensus BFGS ALADIN and Reduced Consensus ALADIN respectively. Finally, to serve the FL community, we compare a variant of Reduced Consensus ALADIN named \texttt{FedALADIN} with existing method in FL. It performs well with several case studies.

	% First we put forward a solution to solve the problem of large-scale data transmission and large-scale consensus QP operation in the T-ALADIN framework. Then we find the relationship between a variant of C-ALADIN named Reduced Consensus ALADIN and Consensus ADMM. Later, the global and local convergence theory of the proposed algorithm is established. 
	
	Other variants of C-ALADIN will be considered in the future to accommodate different types of optimization problems. More importantly, the convergence theory with local inexactly search in C-ALADIN is still lacking.  Such theoretical supplement will assist the algorithm to be applied in more complex neural networks.
	
	\bibliography{paper}

	%\newpage
	
%	\clearpage
	\appendices
	\section{Comparison between   Reduced Consensus ALADIN and Consensus ADMM}\label{sec: Comparison}

	In the standard Consensus ADMM framework, there are two different cases: a) first update the dual then aggregate (update the primal global variable), b) first aggregate  then update the dual. In this subsection we will analyze the Reduced Consensus ALADIN from the above two perspectives. In order to distinguish the Reduced Consensus ALADIN from the Consensus ADMM, we use superscripts on key variables, such as $\lambda_i$ and $z$,  to show difference.

	\subsection{First Update the Dual then Aggregate} 
	Algorithm \ref{alg:FedADMM} shows the first fashion of Consensus ADMM: first update the dual then aggregate \eqref{eq: ADMM1}. In the optimization point of view, \texttt{FedADMM} is a modification of the current fashion.
	% \eqref{eq: ADMM1} 
	\begin{equation}\label{eq: ADMM1}
		\left\{
		\begin{split}
			&{x_i}^+=\mathop{\arg\min}_{x_i}f_i(x_i)+(\lambda_i^{\text{ADMM1}})^\top  x_i+\frac{\rho}{2}\|x_i-z\|^2,\\
			&(\lambda_i^{\text{ADMM1}})^+=(\lambda_i^{\text{ADMM1}})+\rho(x_i^+-z),\\
			&(z^+)^{\text{ADMM1}}=	\frac{1}{N} \sum_{i=1}^{N}\left(x_i^++ \frac{(\lambda_i^{\text{ADMM1}})^+}{\rho}\right).\\
		\end{split}\right.
	\end{equation}
	Here $\lambda_i^\text{ADMM1}$ denotes the dual variables of the current Consensus ADMM fashion.
	From the expression of the subgradients \eqref{eq: subgradient},  one may find 
	\begin{equation}\label{eq: lam-g}
		\lambda_i^\text{ADMM1} = - g_i. 
	\end{equation}

	It can be noticed that the framework of Reduced Consensus ALADIN is very similar to this order of \eqref{eq: ADMM1} . 
	In Reduced Consensus ALADIN, Equation \eqref{eq: consensus QP} boils down to 
	\begin{equation}\label{eq: rhoALADIN}
		\begin{split}
			(z^+_{\rho})^{\text{ALADIN}}=		\mathop{\mathrm{\arg}\mathrm{\min}}_{ \Delta x_i, z}& \quad \mathop{\sum}_{i=1}^{N} \left(\frac{\rho}{2}\Delta x_i^\top \Delta x_i+g_i^\top \Delta x_i\right) \\
			\mathrm{s.t.} &\qquad\Delta x_i+x_i^+=z\; |\lambda_{i}^{\text{ALADIN}}
		\end{split}
	\end{equation}
	which is equivalent to the same operation as that of ADMM framework  (if we ignore the auxiliary variables $\Delta x_i$s)
	\begin{equation}\label{eq: rhoADMM}
		\begin{split}
			(z^+)^{\text{ADMM1}}=&	\mathop{\mathrm{\arg}\mathrm{\min}}_{z} \; \mathop{\sum}_{i=1}^{N} \left(\frac{\rho}{2}\| x_i^+-z\|^2-(\lambda_i^{\text{ADMM1}})^\top z\right).%\\
			%	=&\frac{1}{N} \sum_{i=1}^{N}\left(x_i^++ \frac{(\lambda_i^{\text{ADMM1}})^+}{\rho}\right).
		\end{split}
	\end{equation}
	In both ways of updating the global variable $z$, Equation \eqref{eq: rhoALADIN} and \eqref{eq: rhoADMM}, have the same result as Equation \eqref{eq: reduced QP2}.  %From Equation \eqref{eq: KKT} it's clear that $\Delta x_i=z-x_i^+$ at the optimal solution of each iteration.
	However, the dual update is different in C-ALADIN \eqref{eq: ALADIN dual},
	\begin{equation}\label{eq: ALADIN dual}
		\lambda_{i}^{\text{ALADIN}}=\rho(x_i^--z)-g_i .
	\end{equation}
	Different from Reduced Consensus ALADIN, in the Consensus ADMM iteration \eqref{eq: ADMM1}, 
	\begin{equation}
		\sum_{i=1}^N (\lambda_i^\text{ADMM1})^* =0
	\end{equation}
	is guaranteed only at the optimal point.
	On the opposite, with  Reduced Consensus ALADIN \eqref{eq: ALADIN dual},
	\begin{equation}\label{eq: ALADIN dual KKT}
		\sum_{i=1}^N \lambda_i^{\text{ALADIN}} =0
	\end{equation}
	is guaranteed in each iteration.
	%	Note that the KKT condition here is different from that of T-ALADIN. 
	Since the first version of Consensus ADMM can not bring the latest dual update back to each agent,  the updating rule of the dual in Reduced Consensus ALADIN can be interpreted as a more efficient way for local primal variable update. The Reduced Consensus ALADIN framework makes better use of the consensus dual information.

	%Reduced Consensus ALADIN can be treated as a standard Consensus ADMM framework with an additional step \eqref{eq: ALADIN dual}.   The extra step enforces faster convergence by applying the latest global variable $z$. Moreover, the global variable $z$ will bring different way of convergence proof compare with the typical ALADIN.
	
	\subsection{First Aggregate then Update the Dual}
	
	As introduced in \cite{Boyd2011}, Consensus ADMM can also be interpreted in another fashion:
	\begin{equation}\label{eq: ADMM22}
		\left\{
		\begin{split}
			&{x_i}^+=\mathop{\arg\min}_{x_i}f_i(x_i)+(\lambda_i^{\text{ADMM2}})^\top  x_i+\frac{\rho}{2}\|x_i-z\|^2,\\
			&(z^+)^{\text{ADMM2}}=	\frac{1}{N} \sum_{i=1}^{N}\left(x_i^++ \frac{\lambda_i^{\text{ADMM2}}}{\rho}\right),\\
			&(\lambda_i^{\text{ADMM2}})^+=(\lambda_i^{\text{ADMM2}})+\rho(x_i^+-z^+).
		\end{split}\right.
	\end{equation}
	With \eqref{eq: ADMM22}, Consensus ADMM  can also converge for convex problems with guarantees.
	In this form, the update of $\lambda_i^{\text{ADMM2}}$s has the same property as Reduced Consensus ALADIN that guarantees 
	\begin{equation}\label{eq: ADMM2}
		\sum_{i=1}^N \lambda_i^{\text{ADMM2}} =0
	\end{equation}
	in each iteration. In this way, the dual variables can also carry sensitivity information of the  gap between the latest global variable $z$ and the local variables $x_i^+$. 
	However, the second version of Consensus ADMM can not upload the latest local dual \eqref{eq: lam-g} back to the master because of \eqref{eq: ADMM2}.  Reduced Consensus ALADIN is more efficient for global variable aggregation compare with \eqref{eq: ADMM22}.  %Latest dual information will not be used for aggregation.
	
	Even without Hessian information updated, the Reduced Consensus ALADIN can still inherit both of the benefits of two different Consensus ADMM variations. In fact, it can be treated as a combination of the them. The new designed structure provides both latest information of the dual to the agents and the master.

	\section{ Preliminaries of the Lyapunov Stability Theory}\label{App: Lya}
	
	In this section, we demonstrate the intuition of proving the global convergence of C-ALADIN (Subsection \ref{sec: global convergence} and \ref{sec: convex rate}). Roughly, we adopt the theory of \emph{Lyapunov stability} in control community \cite{nonlinear} for nonlinear systems. The rationale behind such is that the concept of \emph{global convergence} in optimization has a similarity as \emph{stable around the origin} in control.
	
	%		In this paper,  the proof of global convergence (Subsection \ref{sec: global convergence} and \ref{sec: convex rate}) are based on a theory named the \emph{Lyapunov stability} in control community \cite{nonlinear} for nonlinear systems.
	%  {Although optimization and control belongs to different  communities, the basic analysis method can be transfered}.
	%In details, the concept of \emph{global convergence} in optimization has a similarity as \emph{stable around the origin} in control.
	
	We assume that a nonlinear system is described by $\dot\varsigma = \phi(\varsigma)$ where $\varsigma$ is the state variable and $\phi(\cdot)$ is a nonlinear mapping from $\mathbb R^{|\varsigma |}\rightarrow\mathbb R^{|\varsigma |}$ which denotes the dynamic of a system. Here we introduce an unbounded mapping, named  Lyapunov function
	\begin{equation}
		\Phi( \varsigma) : \mathbb R^{|\varsigma |}\rightarrow\mathbb R
	\end{equation}
	$\Phi( \varsigma)$ has the following key properties:
	\begin{itemize}
		\item $\Phi( \varsigma)$ is strictly monotonically decreasing along the trajectories of the system, in other words $\dot \Phi= \nabla_x \Phi^\top  \phi< 0$.
		\item $	\Phi( \varsigma)> 0$ if  $\varsigma \neq 0$; otherwise, $	\Phi( \varsigma)=0$. 
		% if and only if $ \varsigma =0$.
		% \item $	\Phi$  is strictly monotonically decreasing if $ \dot \Phi< 0$.
	\end{itemize}
	% that  $\Phi$  is called positive definite if $	\Phi( \varsigma)\geq 0$ for all $\varsigma$ but  $	\Phi( \varsigma)=0$ if and only is $ \varsigma =0$.  On the other hand, $	\Phi$  is called strictly monotonically decreasing if $ \dot \Phi< 0$.
	With the above properties, we have the following theorem.
	\begin{theorem}\label{theorem: lya}
		If $\Phi$ is positive definite, bounded from below, and strictly monotonically decreasing,
		then $\varsigma$  is bounded and converges to $0$ (stable around the origin) \cite{nonlinear}. 
	\end{theorem}
	
	\begin{proof}
		We next prove that $\varsigma \rightarrow 0$ by showing the following contradiction: 
		
		We assume that $\varsigma$ is not convergent to $0$, then we have $\dot \Phi <0$. From the fundamental theorem of calculus, we then have
		\begin{equation}
			\Phi(\varsigma^0)\geq	\Phi(\varsigma^\infty) = \Phi(\varsigma^0) +\underbrace{\int_{0}^{\infty}  \dot \Phi(\varsigma)}_{<0}=-\infty.
		\end{equation}
		However, this shows a contradiction to the positive definiteness of and bounded below of $\Phi$. Hence we have \[   \Phi(\varsigma)\rightarrow0  \quad\text{which implies}  \quad\varsigma\rightarrow0.\]
	\end{proof}
	% \begin{remark}
		Theorem \ref{theorem: lya} states that the system under analysis is stable when the correlation property of the related Lyapunov function holds.
		Note that the existence of a Lyapunov function is only a sufficient condition to prove the stability around the origin. However we can not confirm that the system is  unstable if we can not find a proper Lyapunov function.
		In our proof, if we treat the optimization algorithms as the nonlinear dynamic $ \phi(\varsigma )$ of the control systems, and the gap between the decision variable and the optimizer as the state $\varsigma$, then the \emph{stable around the origin} analysis is equivalent to prove the \emph{global convergence} of optimization algorithms.

	\section{Proof of Theorem \ref{The: 1}}\label{APP: C}
	\begin{proof}
		
		First we introduce two auxiliary functions $\tilde{f}_i$ and $G(\xi) $:
		\begin{equation}\label{eq: aux}
			\left\{
			\begin{split}
				\tilde{f}_i(\xi_i)&=f_i(\xi_i)+(\lambda_i+\rho(x_i^+-z))^\top \xi_i\\
				G(\xi)&=\sum_{i=1}^{N}f_i(\xi_i)+\sum_{i=1}^{N}(\xi_i-z^*)^\top\lambda_i^*.
			\end{split}\right.
		\end{equation}
		%Here the $\lambda_i^*$s are not requaired to be unique. 
		Assume $f_i$s are strictly convex, then we have
		\begin{equation}\label{eq: aux_ineq}
			\left\{
			\begin{split}
				\tilde{f}_i(x_i^+)&\leq\tilde{f}_i(z^*)-\tilde{\alpha}_i(\|x_i^+-z^*\|)\\
				G(z^*)&\leq G(x^+)-\hat\alpha (\|x_i^+-z^*\|)
			\end{split}\right.
		\end{equation}
		where $ \tilde{\alpha}_i (\cdot)$ and $ \hat{\alpha} (\cdot)$ are class $\mathcal K$ functions.
		If we sum up the first equation of \eqref{eq: aux_ineq},
		\begin{equation}\label{eq: auxiliary 1}
			\begin{split}
				\sum_{i=1}^{N}\left\{f_i(x_i^+)-f_i(z^*)\right\}\leq&\sum_{i=1}^{N}\left\{ ( \lambda_i^++\rho(x_i^+-z))^\top(z^*-x_i^+)\right\}\\
				&-\sum_{i=1}^{N}\tilde{\alpha}_i(\|x_i^+-z^*\|).
			\end{split}
		\end{equation}
		Similarly, from the fact that
		\begin{equation*}
			\begin{split}
				G(z^*)-G(z^+)&\leq0\\
			\end{split}
		\end{equation*}
		we have
		\begin{equation}\label{eq: auxiliary 2}
			\begin{split} \sum_{i=1}^{N}\left\{f_i(z^*)-f_i(x_i^+)\right\}\leq&\sum_{i=1}^{N }(x_i^+-z^*)^\top \lambda_i^*\\
				&-\hat\alpha (\|x_i^+-z^*\|).
			\end{split}
		\end{equation}
		Combine Equations \eqref{eq: auxiliary 1} and \eqref{eq: auxiliary 2},
		we have
		\begin{equation}\label{eq: combine}
			\sum_{i=1}^{N}( \lambda_i^+-\lambda_i^*+\rho( x_i^+-z))^\top(x_i^+-z^*)\leq -\alpha(\|x_i^+-z^*\|)
		\end{equation}
		where \[ \alpha(\|x_i^+-z^*\|)= \hat\alpha (\|x_i^+-z^*\|)+ \sum_{i=1}^{N}\tilde{\alpha}_i(\|x_i^+-z^*\|)\in 
		\mathcal K.\]
		
		The left hand side of Equation \eqref{eq: combine} can be interpreted as			
		\begin{equation}\label{eq: Lya dec}
			\begin{split}
				&\sum_{i=1}^{N}( \lambda_i-\lambda_i^*+\rho(x_i^+-z))^\top(x_i^+-z^*)\\
				\overset{\eqref{eq: sum_lam}}{=}& 		\sum_{i=1}^{N}( \lambda_i-\lambda_i^*)^\top(x_i^+-z)  +\rho(x_i^+-z)^\top(x_i^+-z^*)\\
				\overset{\eqref{eq: primal}}{=}& \sum_{i=1}^{N}\left( \lambda_i-\lambda_i^*+\rho\left(\frac{\lambda_i^+-\lambda_i}{2\rho}+\frac{z^++z}{2} -z^*\right)\right)^\top\\
				&\left(\frac{\lambda_i^+-\lambda_i}{2\rho}+\frac{z^++z}{2} -z\right).
			\end{split}
		\end{equation}	
		By trivially decomposing Equation \eqref{eq: Lya dec}, we get the following six parts
		\begin{itemize}
			\item[Part~1] \begin{equation}
				\begin{split}
					&	\sum_{i=1}^{N} \left(\lambda_i - \lambda_i^*\right)^\top \left( \frac{\lambda_i^+-\lambda_i}{2\rho}\right)\\
					\overset{\eqref{eq: quadratic equality}}{=}&\frac{-1}{4\rho} \sum_{i=1}^{N} \left(\|\lambda_i-\lambda_i^* \|^2 - \|\lambda_i^+-\lambda_i^* \|^2 + \|\lambda_i-\lambda_i^+ \|^2\right).
				\end{split}
			\end{equation}
			\item[Part~2] \begin{equation}
				\begin{split}
					\sum_{i=1}^{N} \rho\left( \frac{\lambda_i^+-\lambda_i}{2\rho}\right)^\top\left( \frac{\lambda_i^+-\lambda_i}{2\rho}\right)
					=\frac{1}{4\rho}\sum_{i=1}^{N}\|\lambda_i^+-\lambda_i\|^2. 
				\end{split}
			\end{equation}
			\item[Part~3] \begin{equation}
				\sum_{i=1}^{N} \rho\left(\frac{z^++z}{2} -z^*\right)^\top   \left( \frac{\lambda_i^+-\lambda_i}{2\rho}\right)\overset{\eqref{eq: sum_lam}}{=}0.
			\end{equation}
			\item[Part~4]  
			\begin{equation}
				\sum_{i=1}^{N} \left(\lambda_i - \lambda_i^*\right)^\top \left(\frac{z^+-z}{2}\right) \overset{\eqref{eq: sum_lam}}{=}0.
			\end{equation}
			\item [Part~5]
			\begin{equation}
				\sum_{i=1}^{N} \rho\left( \frac{\lambda_i^+-\lambda_i}{2\rho}\right)^\top \left(\frac{z^+-z}{2}\right) \overset{\eqref{eq: sum_lam}}{=}0.
			\end{equation}
			\item [Part~6] \begin{equation}
				\begin{split}
					&	\rho N\left(\frac{z^++z}{2} -z^* \right)^\top \left( \frac{z^+-z}{2} \right)\\
					=&	\frac{\rho N}{4} \|z^+-z^*\|^2   -\frac{\rho N}{4} \|z-z^*\|^2.\\
				\end{split}
			\end{equation}
		\end{itemize}
		In the end, the combination of the above six parts gives us:
		\begin{equation}
			\begin{split}
				&\sum_{i=1}^{N}( \lambda_i-\lambda_i^*+\rho(x_i^+-z))^\top(x_i^+-z^*)\\
				=&\frac{1}{4\rho} \sum_{i=1}^{N} \left( -\|\lambda_i-\lambda_i^* \|^2 + \|\lambda_i^+-\lambda_i^* \|^2 \right)\\
				&+	\frac{\rho N}{4} \|z^+-z^*\|^2   -\frac{\rho N}{4} \|z-z^*\|^2\\
				=& \frac{1}{4}\mathscr{L}(z^+,\lambda^+)-\frac{1}{4}\mathscr{L}(z,\lambda)\leq -\alpha(\|x_i^+-z^*\|)\leq 0.
			\end{split}
		\end{equation}
		This says that the introduced Lyapunov function is monotonically decreasing.
	\end{proof}
	
	\section{Proof of Theorem \ref{Them: uniqueness}}\label{APP: D}
	\begin{proof}
		As we have proved, $\mathscr{L}(z,\lambda)\geq \mathscr{L}(z^+,\lambda^+)$ and  $\mathscr{L}(z,\lambda)$ is bounded below. 
		
		We assume that the sequence pair $(x_i, \lambda_i)$ has at least two limit points $(z^\alpha, \lambda_i^\alpha)$ and $(z^\beta, \lambda_i^\beta)$, such that 
		\begin{equation}\left\{
			\begin{split}
				\mathscr{L}^\alpha(z,\lambda)=	\frac{1}{\rho} \sum_{i=1}^{N}\|\lambda_i-\lambda_i^\alpha \|^2 + \rho N  \|z-z^\alpha\|^2\\
				\mathscr{L}^\beta(z,\lambda)=	\frac{1}{\rho} \sum_{i=1}^{N}\|\lambda_i-\lambda_i^\beta \|^2 + \rho N  \|z-z^\beta\|^2.
			\end{split}\right.
		\end{equation}
		We hypothesize  that they may converge to different value
		\begin{equation}\left\{
			\begin{split}
				\lim_{k\rightarrow \infty} \mathscr{L}^\alpha(z,\lambda) = L^\alpha\\
				\lim_{k\rightarrow \infty} \mathscr{L}^\beta(z,\lambda) = L^\beta.
			\end{split}\right.
		\end{equation}
		Note that
		\begin{equation}
			\begin{split}
				&\mathscr{L}^\alpha(z,\lambda)-\mathscr{L}^\beta(z,\lambda)\\
				=&\rho N \left( \| z-z^\alpha \|^2- \| z-z^\beta \|^2\right)\\
				&+\frac{1}{\rho} \sum_{i=1}^{N} \left( \|\lambda_i-\lambda_i^\alpha\|^2-\|\lambda_i-\lambda_i^\beta\|^2\right).
			\end{split}
		\end{equation} 
		For the first part, 
		\begin{equation}
			\begin{split}
				&\rho N \left( \| z-z^\alpha \|^2- \| z-z^\beta \|^2\right)\\
				=&\rho N \left(\left( z^\alpha - z^\beta \right)^\top\left(z^\alpha + z^\beta -2z \right)\right).\\
			\end{split}
		\end{equation}
		Similarly, the second part can be expressed as
		\begin{equation}
			\begin{split}
				&\frac{1}{\rho} \sum_{i=1}^{N} \left( \|\lambda_i-\lambda_i^\alpha\|^2-\|\lambda_i-\lambda_i^\beta\|^2\right)\\
				=&\frac{1}{\rho} \sum_{i=1}^{N} \left(\left( \lambda_i^\alpha - \lambda_i^\beta \right)^\top\left(\lambda_i^\alpha + \lambda_i^\beta -2\lambda_i \right)\right).
			\end{split}
		\end{equation}
		Therefore,
		\begin{equation}
			\begin{split}
				&\mathscr{L}^\alpha(z,\lambda)-\mathscr{L}^\beta(z,\lambda)\\
				=&\rho N \left(\left( z^\alpha - z^\beta \right)^\top\left(z^\alpha + z^\beta -2z \right)\right)\\
				& +\frac{1}{\rho} \sum_{i=1}^{N} \left(\left( \lambda_i^\alpha - \lambda_i^\beta \right)^\top\left(\lambda_i^\alpha + \lambda_i^\beta -2\lambda_i \right)\right).
			\end{split}
		\end{equation} 
		Suppose $(x_i, \lambda_i)$ converges to  $(z^\alpha, \lambda_i^\alpha)$, then we take limit on both sides
		\begin{equation}\label{eq: limit1}
			\begin{split}
				&\mathscr L^\alpha - \mathscr L^\beta \\
				\rightarrow&\rho N \left(\left( z^\alpha - z^\beta \right)^\top\left(z^\alpha + z^\beta -2z^\alpha \right)\right)\\
				&+ \frac{1}{\rho} \sum_{i=1}^{N} \left(\left( \lambda_i^\alpha - \lambda_i^\beta \right)^\top\left(\lambda_i^\alpha + \lambda_i^\beta -2\lambda_i^\alpha \right)\right)\\
				=& - \rho N \left\|z^\alpha-z^\beta \right\|^2-
				\frac{1}{\rho} \sum_{i=1}^{N} \left\|\lambda_i
				^\alpha-\lambda_i^\beta \right\|^2.
			\end{split}
		\end{equation}
		On the opposite, if $(x_i, \lambda_i)$ converges to  $(z^\beta, \lambda_i^\beta)$, then we take another limit
		\begin{equation}\label{eq: limit2}
			\begin{split}
				&\mathscr L^\alpha -\mathscr L^\beta  \rightarrow  \rho N \left\|z^\alpha-z^\beta \right\|^2+
				\frac{1}{\rho} \sum_{i=1}^{N} \left\|\lambda_i
				^\alpha-\lambda_i^\beta \right\|^2.
			\end{split}
		\end{equation}
		From \eqref{eq: limit1} and \eqref{eq: limit2},
		\[\rho N \left\|z^\alpha-z^\beta \right\|^2+
		\frac{1}{\rho} \sum_{i=1}^{N} \left\|\lambda_i
		^\alpha-\lambda_i^\beta \right\|^2=0.\]
		This will induce $z^\alpha=z^\beta$ and $\lambda_i^\alpha=\lambda_i^\beta$. Therefore, the limit point of $(z, \lambda)$ is unique. The optimal solution $(z^*, \lambda_i^*)$ will be obtained if $\mathscr{L}(z,\lambda) = 0$ is touched. 
		
		%As $\mathscr{L}(\cdot)$ is 
		%strictly monotonously decreasing and bounded from below,
		%	it must converge with Reduced Consensus ALADIN when $z$ converges to $z^*$. %The proof established global convergence of Reduced Consensus ALADIN for general convex cases.  
	\end{proof}
	
	\section{Proof of Theorem \ref{them: local model}}\label{APP: E}
	\begin{proof}
		We assume that $x_i^0$s starts from a small neighbor of $z^*$.
		From the first order convergence condition of optimality of Equation \eqref{eq: consensus ALADIN-step1}, one may obtain
		\begin{equation}\label{eq: first order}
			\left\{
			\begin{array}{l}
				\begin{split}
					\nabla f_i(x_i^{+})+\lambda_i +\rho(x_i^{+}-z)&=0\quad \\
					\nabla f_i(z^{*})+ \lambda_i^* &=0.
				\end{split}
			\end{array}
			\right.
		\end{equation}
		Then do minus with the two equations:
		\[\nabla f_i(z^{*})-\nabla f_i(x_i^{+})+(\lambda_i^*-\lambda_i)-\rho(x_i^{+}-z)=0.\] 
		By rearranging the formula:
		\begin{equation}\label{eq: important equality}
			\nabla f_i(z^{*})-\nabla f_i(x_i^{+})+\rho(z^{*}-x_i^{+})=\rho(z^{*}-z)+(\lambda_i-\lambda_i^*).
		\end{equation}
		
		From the \emph{second order sufficient condition} of each sub-problems, there exist a $\sigma$ such that
		\[\left( \nabla^2f_i(x_i^+)+\rho I\right)\succeq \sigma I \succeq0,\] which implies  \[\|\nabla^2f_i(x_i^+)+\rho I\|\succeq \sigma.\]
		In the end, 
		\begin{equation}\label{eq: important inequality}
			\|\nabla f_i(z^{*})-\nabla f_i(x_i^{+})+\rho(z^{*}-x_i^{+})\|\geq\sigma \|z^*-x_i^{+}\|.
		\end{equation}
		Substituting the left hand side by the right hand side of Equation~\eqref{eq: important equality}, then plug into Equation~\eqref{eq: important inequality}:
		\begin{equation}\label{eq: con_ALADIN_step1}
			\begin{split}
				\|\rho(z^{*}-z)+(\lambda_i-\lambda_i^*)\|&\geq \sigma \|z^*-x_i^{+}\|\\
				\Rightarrow	\frac{\rho}{\sigma}\|z-z^{*}\|+\frac{1}{\sigma}\|\lambda_i-\lambda_i^*\|&\geq  \|x_i^{+}-z^*\|.
			\end{split}
		\end{equation}
		
		From Theorem~\eqref{The: 1} and \eqref{Them: uniqueness} we show that the global convergence of $z$ and $\lambda_i$s. Combining with the result of Equation~\eqref{eq: con_ALADIN_step1} we can easily show that the global convergence of $x_i$s.  
	\end{proof}
	
	\section{Proof of Theorem \ref{Them: gradient converge}}\label{APP: F}
	\begin{proof}
		From Equation~\eqref{eq: gradient}, it's can be seen that the expression of $g_i$ consistes two parts $\rho(z-x_i^+)$ and $-\lambda_i$. When both $z$ and $x_i^+$ converge to $z^*$, $\lambda_i$ converge to $\lambda_i^*$,  then it easily can be seen the limit point of $g_i$ is $-\lambda_i^*$.
		
		Now we can show the convergence of $g_i$ with the following equation. 
		\begin{equation}
			\begin{split}
				&\sum_{i=1}^{N} \left\| g_i-g_i^* \right\| \\
				\overset{\eqref{eq: gradient}}{=}	&\sum_{i=1}^{N}  \left\|\rho(z-x_i^+)-\lambda_i+\lambda_i^* \right\| \\
				\leq& \rho N \left\|z-z^*\right\|+ \sum_{i=1}^{N}\rho \left\|  x_i^+-z^*\right\|  + \left\|\lambda_i-\lambda_i^*\right\|.
			\end{split}
		\end{equation}
		
		From Theorem~\eqref{The: 1}  and \eqref{them: local model}, all of the three parts converge globally, therefore $ \sum_{i=1}^{N} \left\| g_i+\lambda_i^* \right\| $ will also converge to $ 0$ globally.
	\end{proof}
	
	\section{Proof of Theorem \ref{theorem 3 }}\label{APP: G}
	\begin{proof}
		Suppose $\sum_{i=1}^{N} f_i(x_i)$ is $m_f$ strongly convex, then the following inequality holds \cite{ling2015dlm}
		\begin{equation}\label{eq: linear rate}
			\begin{split}
				m_f \sum_{i=1}^{N}&\left\| x_i^+ -z^*\right\|^2 
				\leq \sum_{i=1}^{N} \left(x_i^+-z^* \right)^\top \left( g_i -g_i^* \right)\\
				\overset{\eqref{eq: gradient}}{=}& \sum_{i=1}^{N} \left(x_i^+-z^* \right)^\top \left( \rho \left( z-x_i^+\right) -\lambda_i+\lambda_i^* \right)\\
				\overset{\eqref{eq: primal}}{=}& \sum_{i=1}^{N} \left( \frac{\lambda_i^+-\lambda_i}{2\rho}+\frac{z^++z}{2} -z^* \right)^\top \\
				&\left( \rho \left( z-\frac{\lambda_i^+-\lambda_i}{2\rho}-\frac{z^++z}{2} \right) -\lambda_i+\lambda_i^* \right).\\
			\end{split}
		\end{equation}
		The above equation can be broken into nine parts:
		\begin{itemize}
			\item [Part~1]
			\begin{equation}
				\begin{split}
					\sum_{i=1}^{N} \left( \frac{\lambda_i^+-\lambda_i}{2\rho}\right)^\top\frac{\rho}{2}(z-z^+)\overset{\eqref{eq: sum_lam}}{=}0.
				\end{split}
			\end{equation}

			\item [Part~2]
			\begin{equation}
				\begin{split}
					&\sum_{i=1}^{N} \left( \frac{\lambda_i^+-\lambda_i}{2\rho}\right)^\top \left(\frac{\lambda_i-\lambda_i^+}{2} \right)\\
					=&\; -\frac{1}{4\rho}\sum_{i=1}^{N} \|\lambda_i-\lambda_i^+\|^2.
				\end{split}
			\end{equation}
			
			\item [Part~3]\begin{equation}
				\begin{split}
					&	\sum_{i=1}^{N}  \left( \frac{\lambda_i^+-\lambda_i}{2\rho}\right)^\top  \left( \lambda_i^*-\lambda_i\right)  \\
					\overset{\eqref{eq: quadratic equality}}{=}&\frac{1}{4\rho} 	\sum_{i=1}^{N}\left( \|\lambda_i-\lambda_i^*\|^2  -\|\lambda_i^+-\lambda_i^*\|^2\right)\\
					&+\frac{1}{4\rho} 	\sum_{i=1}^{N}\|\lambda_i-\lambda_i^+\|^2.
				\end{split}
			\end{equation}
			
			\item [Part~4]\begin{equation}
				\begin{split}
					&\sum_{i=1}^{N} \left(\frac{z^++z}{2}\right)^\top \frac{\rho}{2}(z-z^+)\\
					=\;& \frac{\rho N}{4} \left( \|z\|^2-\|z^+\|^2\right).
				\end{split}
			\end{equation} 
			
			\item [Part~5]\begin{equation}\sum_{i=1}^{N} \left(\frac{z^++z}{2}\right)^\top\left(\frac{\lambda_i-\lambda_i^+}{2}\right)\overset{\eqref{eq: sum_lam}}{=}0. \end{equation}
			
			\item [Part~6]  \begin{equation} \sum_{i=1}^{N} \left(\frac{z^++z}{2}\right)^\top\left( \lambda_i^*-\lambda_i\right)  \overset{\eqref{eq: sum_lam}}{=}0. 
			\end{equation}
			
			\item [Part~7]\begin{equation}
				\begin{split}
					\sum_{i=1}^{N}\left(-z^*\right)^\top\frac{\rho}{2}(z-z^+)
					=\frac{\rho N}{2} \left( \left(z^+ \right)^\top z^* - z^\top z^*\right).
				\end{split}
			\end{equation}
			
			\item [Part~8]\begin{equation} \sum_{i=1}^{N} \left(-z^*\right)^\top \left(\frac{\lambda_i-\lambda_i^+}{2} \right)\overset{\eqref{eq: sum_lam}}{=}0.
			\end{equation}
			
			\item [Part~9] \begin{equation} \sum_{i=1}^{N} \left(-z^*\right)^\top \left( \lambda_i^*-\lambda_i\right)\overset{\eqref{eq: sum_lam}}{=}0.
			\end{equation}
		\end{itemize}
		In the end,  combine the above nine parts,
		\begin{equation}\label{eq: mf}
			\begin{split}
				m_f \sum_{i=1}^{N}&\left\| x_i^+ -z^*\right\|^2 
				\leq \sum_{i=1}^{N} \left(x_i^+-z^* \right)^\top \left( g_i -g_i^* \right)\\
				=&  \frac{1}{4\rho}\sum_{i=1}^{N} \left( \|\lambda_i-\lambda_i^*\|^2  -\|\lambda_i^+-\lambda_i^*\|^2\right) \\
				&+ \frac{\rho N}{4} \left( \|z\|^2  -2z^\top z^* +\|z^*\|^2 \right)\\
				&-\frac{\rho N}{4} \left( \|z^+\|^2  -2\left( z^+\right)^\top z^* +\|z^*\|^2 \right)\\
				=&\frac{1}{4}\mathscr{L}(z,\lambda)-\frac{1}{4}\mathscr{L}(z^+,\lambda^+).
			\end{split}
		\end{equation}
		
		Following the assumption in Theorem \ref{theorem 3 }, if 
		\[  \delta \mathscr{L}(z^+,\lambda^+)\leq 4 m_f \sum_{i=1}^{N}\left\| x_i^+ -z^*\right\|^2 \] holds, with the result of \eqref{eq: mf},
		we have	\[\mathscr{L}(z^+,\lambda^+)\leq \frac{1}{1+\delta}\mathscr{L}(z,\lambda).\]
		Later, a serial recurrence formula can be established:
		\begin{equation}
			\begin{split}
				\rho N  \|z-z^*\|^2 &\leq\frac{1}{\rho} \sum_{i=1}^{N}\|\lambda_i-\lambda_i^* \|^2 + \rho N  \|z-z^*\|^2	\\
				&= \mathscr{L}(z,\lambda)\leq \left( \frac{1}{1+\delta} \right)^k 	\mathscr{L}(z^0,\lambda^0).
			\end{split}
		\end{equation} 
		Therefore we have Q-linearly converge of the global variable $z$.
		\begin{equation}
			\|z-z^*\| \leq \frac{1}{\sqrt{\rho N}} \left( \frac{1}{\sqrt{1+\delta}} \right)^k 	\mathscr{L}(z^0,\lambda^0)^{\frac{1}{2}}.
		\end{equation} 
		Moreover, on the dual side we have
		\begin{equation}
			\begin{split}
				\frac{1}{\rho} \sum_{i=1}^{N}\|\lambda_i-\lambda_i^* \|^2 &\leq\frac{1}{\rho} \sum_{i=1}^{N}\|\lambda_i-\lambda_i^* \|^2 + \rho N  \|z-z^*\|^2\\
				&	= \mathscr{L}(z,\lambda)\leq \left( \frac{1}{1+\delta} \right)^k 	\mathscr{L}(z^0,\lambda^0),
			\end{split}
		\end{equation}
		which induce the convergence of the dual
		\begin{equation}
			\sum_{i=1}^{N}\|\lambda_i-\lambda_i^* \|^2 \leq \rho  \left( \frac{1}{1+\delta} \right)^k 	\mathscr{L}(z^0,\lambda^0).
		\end{equation}
		This shows a Q-linear convergence rate of Reduced Consensus ALADIN.
		
	\end{proof}
	
	\section{Proof of Theorem \ref{them: local convergence}}\label{APP: H}
	\begin{proof}
		First, let the result of Equation~\eqref{eq: first order}-\eqref{eq: con_ALADIN_step1} holds, four different situations of local convergence analysis of C-ALADIN can be established by combing with SQP theory \cite[Chapter 18]{Nocedal2006}. First, we show the local convergence analysis of Reduced Consensus ALADIN. As an extension, the corresponding analysis of Consensus BFGS ALADIN is also discussed. %Appendix~\ref{appendix}.

		%We have four different situations to be analyzed.
		As mentioned in Section \ref{sec: conALADIN}, Reduced Consensus ALADIN is proposed by setting $B_i=\rho I$. 
		%	If $\nabla^2f_i(x_i)\rightarrow \rho I$, 
		Assume $\gamma$ is the upper bound of the Hessian approximation distance,  $\left\|\rho I-\nabla^2f_i(x_i)\right\|\leq \gamma$,
		the convergence rate analysis can be divided into four situations.
		From the SQP theory of Newton Lagrange method \cite{nonlinear-programming}, we have different convergence rate guarantees for the first three situations of the following four. 
		
		a) If $\left\|\rho I-\nabla^2f_i(x_i)\right\|\leq \gamma \leq 1$, 
		we have
		\begin{equation}\label{eq: step2-linear}
			\left\{
			\begin{split}
				N\|z^{+}-z^*\|&\leq \gamma\sum_{i=1}^{N}\|x_i^{+}-z^*\|\\ \sum_{i=1}^{N}\|\lambda_i^{+}-\lambda_i^*\|&\leq \gamma\sum_{i=1}^{N}\|x_i^{+}-z^*\|.
			\end{split}\right.
		\end{equation} 
		Conbine Equation \eqref{eq: con_ALADIN_step1} and \eqref{eq: step2-linear}, the following inequality can be obtained.
		\begin{equation}
			\begin{split}
				&\left( 	\frac{\rho N}{\sigma}\|(z^+-z^*)\|+\frac{1}{\sigma}\sum_{i=1}^{N}\|(\lambda_i^{+}-\lambda_i^*)\|\right) \\
				\leq &\frac{\left( \rho +1\right)\gamma}{\sigma}\sum_{i=1}^{N}\|x_i^{+}-z^*\|\\
				\overset{\eqref{eq: con_ALADIN_step1}}{\leq} &\frac{\left( \rho +1\right)\gamma}{\sigma}\left( 	\frac{\rho N}{\sigma}\|z-z^{*}\|+\frac{1}{\sigma}\sum_{i=1}^{N}\|\lambda_i-\lambda_i^*\|\right).  
			\end{split}
		\end{equation}
		This shows a local linear convergence of Reduced Consensus ALADIN.

		b) If $\gamma\rightarrow 0$, we have
		\begin{equation}
			\begin{split}
				&\left( 	\frac{\rho N}{\sigma}\|(z^+-z^*)\|+\frac{1}{\sigma}\sum_{i=1}^{N}\|(\lambda_i^{+}-\lambda_i^*)\|\right)\\
				\leq& \frac{\left( \rho +1\right)\omega_f}{2\sigma}\left( \sum_{i=1}^{N}\|x_i^{+}-z^*\|\right)^2\\
				\overset{\eqref{eq: con_ALADIN_step1}}{\leq}
				&\frac{\left( \rho +1\right)\omega_f}{2\sigma}\left( 	\frac{\rho N}{\sigma}\|z-z^{*}\|+\frac{1}{\sigma}\sum_{i=1}^{N}\|\lambda_i-\lambda_i^*\|\right)^2, 
			\end{split}
		\end{equation} where $\omega_f$ denotes the corresponding Lipschitz constant \cite{Nocedal2006}.
		This shows a quadratic convergence of C-ALADIN. 
		
		c) If $\kappa=\frac{\left( \rho +1\right)}{\sigma}   \left(\gamma+\frac{\omega_f}{2}\sum_{i=1}^{N}\|x_i^{+}-z^*\| \right)\rightarrow0$,  we have
		\begin{equation}
			\begin{split}
				&\left( 	\frac{\rho N}{\sigma}\|(z^+-z^*)\|+\frac{1}{\sigma}\sum_{i=1}^{N}\|(\lambda_i^{+}-\lambda_i^*)\|\right) \\
				\leq &\frac{\left( \rho +1\right)}{\sigma}   \left(\gamma+\frac{\omega_f}{2}\sum_{i=1}^{N}\|x_i^{+}-z^*\| \right)\left(\sum_{i=1}^{N}\|x_i^{+}-z^*\| \right) \\
				= &\;\kappa\left(\sum_{i=1}^{N}\|x_i^{+}-z^*\| \right) \\
				\overset{\eqref{eq: con_ALADIN_step1}}{\leq} &\;\kappa\left( 	\frac{\rho N}{\sigma}\|z-z^{*}\|+\frac{1}{\sigma}\sum_{i=1}^{N}\|\lambda_i-\lambda_i^*\|\right).  
			\end{split}
		\end{equation}
		Reduced Consensus ALADIN will show a local superlinear convergence.

		d)
		If $\gamma$ is large, the above analysis can not be applied. 
		Luckily, if $\rho$ is sufficient large, Equation \eqref{eq: consensus ALADIN-step1} is locally strongly convex.
		%\begin{equation}\label{eq: auxfun}
		%\begin{split} 
		%		\min_{x, z}\;\;& \sum_{i=1}^{N}f_i(x_i)+\frac{\rho}{2}\|x_i-z\|^2 \\ \quad\mathrm{s.t.}\;\;& x_i = z\  |\lambda_i^+. \;
		%	\end{split}
	%	\end{equation}
Corollary \eqref{corollary} can be prepared here with different $\delta^k$ in each iteration $k$.  Therefore, assume there exist a $\bar \delta \geq \delta^k$ and satisfies Equation \eqref{eq: strongly convex}, then Reduced Consensus ALADIN can still get Q-linearly converge with rate $\left( \frac{1}{\sqrt{1+\bar\delta}} \right)$.

By recovering the Hessian approximation $B_i$ with Equation \eqref{eq: BFGS}, Consensus BFGS ALADIN 
can touch linear or super-linear convergence rate, since $\|B_i- \nabla^2 f_i(x_i)\|\rightarrow 0$ can easily hold.
With the technologies of \emph{Sharpen BFGS} \cite{jin2022sharpened}, super-linear or even quadratic rate can be obtained.
\end{proof}

\begin{remark}
Although the SQP type algorithms has local convergence guarantees for non-convex optimization problems, the performance depends on the initial point and the way of combined globalization technologies. Although ALADIN inherits the local convergence properties of SQP, the consensus globalization technologies (i.e. Armijo line search or trust region method \cite{Nocedal2006}) are not needed.
\end{remark}
%	\section{}
%	Appendix two text goes here.
%	
%	
%	% use section* for acknowledgment
%	\ifCLASSOPTIONcompsoc
%	% The Computer Society usually uses the plural form
%	\section*{Acknowledgments}
%	\else
%	% regular IEEE prefers the singular form
%	\section*{Acknowledgment}
%	\fi
%	
%	
%	The authors would like to thank...

\ifCLASSOPTIONcaptionsoff
\newpage
\fi



%\bibliographystyle{IEEEtran}
%	\begin{thebibliography}{1}
%		
%		\bibitem{IEEEhowto:kopka}
%		H.~Kopka and P.~W. Daly, \emph{A Guide to {\LaTeX}}, 3rd~ed.\hskip 1em plus
%		0.5em minus 0.4em\relax Harlow, England: Addison-Wesley, 1999.
%		
%	\end{thebibliography}

\bibliographystyle{unsrt}


%	Biography text here.
%\end{IEEEbiography}
%	Biography text here.

%\newpage

%	Biography text here.


%\vfill


\end{document}